\documentclass[12pt]{amsart}

\usepackage{hyperref}
\hypersetup{nesting=true,debug=true,naturalnames=true}
\usepackage{graphicx,amssymb,upref}

\usepackage{amsmath,amsthm}
\usepackage{amsfonts}
\usepackage{latexsym}
\usepackage{amsbsy}
\usepackage{amssymb,mathscinet}
\numberwithin{equation}{section}

\usepackage{amsfonts,amssymb,amscd,amsmath,enumerate,verbatim,calc,enumerate}
\theoremstyle{plain}
\newtheorem{theorem}{Theorem}[section]
\newtheorem{corollary}[theorem]{Corollary}
\newtheorem{lemma}[theorem]{Lemma}
\newtheorem{proposition}[theorem]{Proposition}

\newtheorem{notation}[theorem]{Notation}
\newtheorem{definition}[theorem]{Definition}

\newtheorem{remark}[theorem]{Remark}
\newtheorem{example}[theorem]{Example}

\newcommand{\ot}{\otimes}

\newcommand{\wt}{\widetilde}
\newcommand{\wT}{\wt{T}}

\newcommand{\wV}{\wt{V}}

\begin{document}
	
\title[Cauchy dual and Wold-type decomposition]
{Cauchy dual and Wold-type decomposition for bi-regular covariant representations}

	\date{\today}
	
	\author[Saini]{Dimple Saini}
\address{Centre for Mathematical and Financial Computing, Department of Mathematics, The LNM Institute of Information Technology, Rupa ki Nangal, Post-Sumel, Via-Jamdoli
	Jaipur-302031,
	(Rajasthan) India}
\email{18pmt006@lnmiit.ac.in,  dimple92.saini@gmail.com}

	\begin{abstract}
	The notion of Cauchy dual for left-invertible covariant representations was studied by Trivedi and Veerabathiran. Using the Moore-Penrose inverse, we extend this notion for the covariant representations having closed range and explore several useful properties. We obtain a Wold-type decomposition for {regular} completely bounded covariant representation whose Moore-Penrose inverse is regular. Also, we discuss an example related to the non-commutative bilateral weighted shift. We prove that the Cauchy dual of the concave covariant representation $(\sigma, V)$ modulo $N(\wV)$ is hyponormal modulo $N(\wV)$.
	\end{abstract}



	\maketitle		
		
\section{Introduction}
The fundamental theorem of Wold \cite{W38} says that: Every isometry on a Hilbert space is either a unitary, or a shift, or uniquely decomposes as a direct summand of both. In \cite{AB49}, Beurling proved each closed $M_{z}$-invariant subspace of the Hardy space $H^{2}(\mathbb{D})$ is the range of an inner operator, this result is an application of the Wold decomposition. A generalization of Beurling’s theorem is derived by Halmos \cite{H61} as the wandering subspace theorem for the shift operator. Richter \cite{R88} considered the class of pure concave operators and proved a wandering subspace theorem. The concave operators are also known as $2$-expansive operators where the notion of $n$-expansive operators is due to Aleman \cite{A93} and Athavale \cite{A96}:
\begin{definition}
	For $n\in \mathbb{N},$ a bounded linear operator $V$ on a Hilbert space $\mathcal{H}$ is called {\rm $n$-expansive} if $$\sum_{i=0}^{n} (-1)^i \binom{n}{i} V^{*i}V^i \le 0.$$ %If $V$ is $n$-expansive for all $n\ge 1,$ then we say that $V$ is {\rm completely hyperexpansive}.
\end{definition} 
\noindent Shimorin \cite{SS01} gave a Wold-type decomposition for the concave operators and obtained an elementary proof of Richter's wandering subspace theorem. Olofsson \cite{O05} identified a growth condition and extended Richter’s wandering subspace theorem in the following way: 
\begin{theorem}$($Olofsson$)$
	Let $V$ be a bounded linear operator on a Hilbert space $\mathcal{H}$ such that
	\begin{itemize}
		\item [(i)] $V$ is pure expansive;
		\item [(ii)] $V$ satisfies a growth condition that there exist some positive numbers $d, d_{m}$  such that $\sum_{m \geq 2 } \frac{1}{d_{m}} =\infty$ and \begin{equation*}
			\|V^{m}h\|^{2}\leq  d_{m}(\|Vh\|^{2}-\|h\|^{2})+d\|h\|^{2}  , \quad h \in \mathcal{H}.
		\end{equation*} 
	\end{itemize}
	Then $V$ has the wandering subspace property.
\end{theorem}

\noindent Olofsson's wandering subspace theorem extended by Ezzahraoui, Mbekhta, and Zerouali in \cite{EMZ15} where they considered the operators $V$ with reduced minimum modulus greater than or equal to $1$ (this generalizes the notion of expansive operators), see the following definition and theorem for more details:
\begin{definition}
	Let $V$ be a bounded linear operator on a Hilbert space $\mathcal{H}.$
	\begin{itemize}
		\item[(1)] $V$ is called {\rm regular} if range of $V,$ $R(V),$ is closed and $ker(V)\subseteq V^n\mathcal{H}$ for all $n\ge 0;$
		\item[(2)] The {\rm reduced minimum modulus} for $V$ is defined by $$	\gamma({V}):=\begin{cases}
			\inf\{\|Vh\|; \|h\|=1, h\in {N(V)^{\perp}}\}  & \text{if }   V\neq 0 \\
			\infty,  & \text{if }  {V} = 0.
		\end{cases}$$ 
	\end{itemize}
\end{definition}

\begin{theorem}\label{1234}$($Ezzahraoui-Mbekhta-Zerouali$)$
	Let $V$ be a regular operator on a Hilbert space $\mathcal{H},$ and let ${V}^{\dagger}$ be the Moore-Penrose inverse of $V$ (that is, ${V}{V}^{\dagger}{V}={V}, {V}^{\dagger}{V}{V}^{\dagger}={V}^{\dagger}, ({V}{V}^{\dagger})^*={V}{V}^{\dagger},( {V}^{\dagger}{V})^*={V}^{\dagger}{V}$ ) such that \begin{align*}
		\|V^{n}h\|^{2}\leq d_{n}(\|Vh\|^{2}-\|V^{\dagger}Vh\|^{2})+\|V^{\dagger}Vh\|^{2},\:\:\: h \in \mathcal{H}
	\end{align*} with  $\gamma(V)\ge 1$ and $\sum_{ n\geq 2}\frac{1}{d_{n}}=\infty.$ Then $\mathcal{H}= \bigvee_{n\geq0} V^n(\mathcal{E})+\bigcap_{n=0}^{\infty}V^{n}\mathcal{H}$ with $\mathcal{E}=\mathcal{H}\ominus V\mathcal{H}.$
\end{theorem}\noindent  Further, in \cite{EMZ21}, Ezzahraoui, Mbekhta, and Zerouali discussed a Wold-type decomposition for regular operators with regular Moore-Penrose inverse. 

The notion of Cauchy dual for left-invertible operators was introduced by Shimorin \cite{SS01}. Ezzahraoui, Mbekhta, and Zerouali \cite{EMZ19} extended the notion of Cauchy dual to operators with closed range. Using the Cauchy dual, Chavan \cite{SC07} proved the following theorem:  
\begin{theorem}$($Chavan$)$
	Suppose that $V$ is a concave operator on a Hilbert space $\mathcal{H}.$ Then the Cauchy dual $V'$ of $V$ is a hyponormal contraction. 
\end{theorem} Cuntz \cite{C77} studied a $C^*$-algebra, known as Cuntz algebra, generated by isometries $V_1,\dots,V_n$  $(n\ge 2)$ on a Hilbert space with  $\sum_{i=1}^{n}V_iV_i^*=I,$ or equivalently, $V_i$'s have orthogonal range. Wold decomposition for two isometries with orthogonal range was first introduced by Frazho \cite{F84}. Popescu \cite{POP89} extended this decomposition to the case of an infinite sequence of isometries with orthogonal final spaces. 

Using the notion of isometric representations of $C^*$-correspondences, Pimsner \cite{P97} generalized the construction of Cuntz algebras. Muhly and Solel \cite{MS99} presented the Wold decomposition for the isometric covariant representations of tensor algebras of $C^*$-correspondences based on \cite{POP89}. Trivedi and Veerabathiran \cite{HS19} proved a Halmos-Richter-type wandering subspace theorem for concave covariant representations of $C^*$-correspondences and introduced the notion of Cauchy dual for left-invertible covariant representations of $C^*$-correspondences. Rohilla, Veerabathiran, and Trivedi \cite{AHS22} extended this decomposition and generalized Theorem \ref{1234} for regular covariant representations having reduced minimum modulus $\gamma({\wV})\ge 1$ that satisfy corresponding growth condition. The main purpose of this paper is to follow this direction along \cite{EMZ21} and to investigate the conditions when the wandering subspace theorem fails, and to analyze the Cauchy dual of concave and bi-regular covariant representations of $C^*$-correspondences.

The section-wise plan is as follows: In Section 2, we recall the definition of {generalized inverse} for completely bounded covariant representation and derive its properties. In Section 3, we obtain the notion of Cauchy dual of covariant representations and a Wold-type decomposition for regular completely bounded covariant representation with regular Moore-Penrose inverse. In Section 4, we study the powers of the Moore-Penrose inverse of completely bounded covariant representations. In Section 5, we discuss the Cauchy dual of concave completely bounded covariant representations based on \cite{EMZ19}.

\subsection{Preliminaries and Notations}
Here, we will review some basic concepts from \cite{L95, MS98, MS99, P97}.	
Suppose that $E$ is a Hilbert $C^*$-module over a $C^*$-algebra $\mathcal{B}.$ It is called $C^*$-correspondence over $\mathcal{B}$ if $E$ has a left $\mathcal{B}$-module structure by a non-zero $*$-homomorphism $\phi:\mathcal B\to \mathcal L(E),$ that is, $ b\xi :=\phi(b)\xi$ for all $b\in\mathcal B$ and $\xi\in E,$ where $\mathcal L(E)$ is the collection of all adjointable operators on $E.$ In this paper, each $*$-homomorphism considered is nondegenerate. Throughout this paper, we will use the following notations: $\mathcal{H}$ for a Hilbert space, $E$ for a $C^*$-correspondence over $\mathcal{B}$ and $B(\mathcal{H})$ for the algebra of bounded linear operators on $\mathcal{H}.$ For $V\in B(\mathcal{H}),$ we denote by $R(V)$ and $N(V)$ the range of $V$ and the kernel of $V,$ respectively. The following notion of completely bounded covariant representations plays an important role in this paper:

\begin{definition}
	Let $V:E\to B(\mathcal H)$ be a linear map and $\sigma:\mathcal B\to B(\mathcal H)$ be a representation. Then the pair $(\sigma,V)$ is
	called a {\rm  covariant representation} (cf. \cite{MS98}) of $E$ on $\mathcal H$  if 
	\[
	V(b\xi c)=\sigma(b)V(\xi)\sigma(c) \quad \quad (\xi\in E,
	b,c\in\mathcal B).
	\] 
	We say that $(\sigma,V)$ is {\rm completely bounded covariant representation} (simply say, {\rm c.b.c. representation})  if $V$ is completely bounded. Further, $(\sigma, V)$ is called {\rm isometric} if $V(\xi)^*V(\eta)=\sigma(\langle \xi,\eta \rangle)$ for all $\xi,\eta\in E.$
\end{definition}

The following result is due to Muhly and Solel \cite{MS98} which is useful to classify the c.b.c. representations of a $C^*$-correspondence:

\begin{lemma} \label{MSC}
	The function $(\sigma, V)\mapsto \wV$ provides a bijection between the set of all c.b.c. (respectively, completely contractive covariant) representations $(\sigma, V)$ of $E$ on $\mathcal H$ and the set of all bounded (respectively,
	contractive) linear maps $\widetilde{V}:E\otimes_{\sigma} \mathcal H\to \mathcal
	H$ defined by
	\[
	\widetilde{V}(\eta\otimes h)=V(\eta)h \quad \quad (
	h\in\mathcal H,\eta\in E),
	\] such that $\sigma(b)\widetilde{V}=\widetilde{V}(\phi(b)\otimes I_{\mathcal
		H})$ for all $b\in\mathcal B$. Moreover, $\wV$ is {\rm isometry} if and only if $(\sigma, V)$ is isometric.
\end{lemma}

Let $(\sigma , V)$ and $(\psi, T)$ be two c.b.c. representations of ${E}$ on the Hilbert spaces $\mathcal{H}$ and $\mathcal{K},$ respectively. We say that $(\sigma , V)$ is {\it isomorphic} to $(\psi, T)$ if there exists a unitary $U:\mathcal H \to \mathcal K$ such that   $U\sigma(b)=\psi(b)U$ and $U\wV=\wT(I_{E}\ot U)$ for all $b\in \mathcal{B}.$

\begin{definition}	
	Let $(\sigma,V)$  be a  c.b.c. representation of ${E}$ on $\mathcal{H}.$ A closed  subspace $\mathcal{K}$ of $\mathcal{H}$ is called  $(\sigma,V)$-{\rm invariant} $(resp. (\sigma,V)$-{\rm reducing}) (cf. \cite{HS19}) if it is   $\sigma(\mathcal{B})$-invariant and (resp. both $\mathcal{K},\mathcal{K}^{\perp}$) is invariant by each operator  $V(\xi)$ for all $\xi \in E.$ %The restriction provides a new representation $(\sigma , V)|_{\mathcal{K}}$ of $E$ on $\mathcal{K}.$
\end{definition}
For each $n\in \mathbb{N}_0(=\mathbb{N}\cup \{0\})$,
$E^{\otimes n} =E\otimes_{\phi} \dots \otimes_{\phi}E$ ($n$-times) (here $E^{\otimes 0} =\mathcal{B}$) is a $C^*$-correspondence over $\mathcal{B}$, with the left module action  of $\mathcal{B}$ on $E^{\otimes n}$ defined as  $$\phi_n(b)(\xi_1 \otimes \dots \otimes \xi_n)=b\xi_1\otimes \dots \otimes\xi_n, \: \:\:\: b\in \mathcal{B},\xi_i \in E.$$
For  $n\in \mathbb{N},$ define $\wV_n : E^{\ot n}\ot \mathcal{H} \to \mathcal{H}$ by $$\wV_n (\xi_1 \ot \dots \ot \xi_n \ot h) = V (\xi_1) \dots V(\xi_n) h, \quad \xi_i \in E, h \in \mathcal H.$$ 
The \emph{Fock module} of $E$ (cf. \cite{F02}), $\mathcal{F}(E)= \bigoplus_{n \geq 0}E^{\otimes n},$ is a $C^*$-correspondence over $\mathcal{B},$ where the left module action  of $\mathcal{B}$ on $\mathcal{F}(E)$ is defined  by $$\phi_{\infty}(b)\left(\oplus_{n \geq 0}\xi_n\right)=\oplus_{n \geq 0}\phi_n(b)\xi_n , \:\: \xi_n \in E^{\otimes n}.$$
For $\xi \in E,$  the \emph{creation operator} $V_{\xi}$ on $\mathcal{F}(E)$ is defined by $$V_{\xi}(\eta)=\xi \otimes \eta, \:\: \eta \in E^{\otimes n}, n\ge 0.$$ 

\begin{theorem} $($Pimsner \cite{P97}$)$
	Let $\mathcal{T}{(E)}$ be a $C^*$-algebra in $\mathcal{L}(\mathcal{F}(E)),$ called {\rm Toeplitz algebra}, generated by $\{V_{\xi}\}_{\xi\in E}$ and $\{\phi_{\infty}(b)\}_{b\in \mathcal{B}}.$ 
	Let $(\sigma,V)$ be an isometric covariant representation of $E$ on $\mathcal H.$ Then the map $$\begin{cases}
		V_{\xi}\mapsto V(\xi), &   \xi\in E; \\
		\phi_{\infty}(b)\mapsto \sigma(b), &   b\in \mathcal{B}
	\end{cases}$$ extends uniquely to a $C^*$-representation of $\mathcal{T}(E)$ on H. Conversely, suppose that $\pi: \mathcal{T}(E)\mapsto B(\mathcal{H})$ is a $C^*$-representation, and if $V(\xi)$ is defined to be $\pi(V_{\xi})$ for $\xi\in E,$ and $\sigma(b)$ is defined to be $\pi(\phi_{\infty}(b))$ for $b\in \mathcal{B},$ then $(\sigma,V)$ is an isometric covariant representation of $E$ on $\mathcal{H}.$
\end{theorem}

Now, we recall the following definition of regular c.b.c. representation from \cite{AHS22} which extends the notion of isometric covariant representations: %Let $X$ be a bounded linear operator between two Hilbert spaces. We denote by $R(X)$ and $N(X),$ the range of $X$ and the kernel of $X,$ respectively. 

\begin{definition}
	Let $(\sigma,V)$ be a c.b.c. representation of ${E}$ on $\mathcal H.$ We say that $(\sigma,V)$ is {\rm regular} if $N(\widetilde{V})\subseteq {E} \otimes{ {{R}^{\infty}{({\wV})}}}$ and its range $R(\widetilde{V})$ is closed. Here ${ {{R}^{\infty}{({\wV})}}}:=\bigcap_{n \geq 0} { {{R}{(\widetilde{V}_n)}}}$ is the {\rm generalized range} of $(\sigma,V).$
\end{definition}

Suppose that $(\sigma,V)$ is a left-invertible c.b.c. representation of $E$ on $\mathcal{H}$ (that is, $(\wV^*\wV)^{-1}\wV^*$ is a left inverse of $\wV$). Then $R(\widetilde{V})$ is closed and $N(\wV)=\{0\},$ and hence $(\sigma,V)$ is regular. Therefore, every left-invertible c.b.c. representations are regular.  

The following result is from \cite[Theorem 2.2]{AHS22} which characterizes the regular c.b.c. representations:
\begin{theorem}\label{cc}
	Let $(\sigma,V)$ be a c.b.c. representation of ${E}$ on $\mathcal H.$ Then, for each  $ m, n \in \mathbb{N} $, the following statements are equivalent:
	\begin{enumerate}
		\item $ N(\widetilde{V}) \subseteq (I_{E} \otimes \widetilde{V}_m)(E^{\otimes(m+1) }\otimes \mathcal{H})$  ;
		\item $ N({\widetilde{V}}_n) \subseteq (I_{E^{\otimes n}} \otimes \widetilde{V})(E^{\otimes(n+1) }\otimes \mathcal{H})$ ;
		\item $ N({\widetilde{V}}_n) \subseteq(I_{E^{\otimes n}} \otimes \widetilde{V}_m)(E^{\otimes(n+m) }\otimes \mathcal{H})$  ;
		\item $N(\widetilde{V}_{n}) = (I_{E^{\otimes n}} \otimes \widetilde{V}_m)N (\widetilde{V}_{m+n})$.
	\end{enumerate}
	
\end{theorem}

\section{Properties of generalized inverse of regular covariant representations}
In this section, we will discuss the definition of {\it generalized inverse} and derive some natural properties of the generalized inverse. Also, we study the necessary and sufficient conditions for regular covariant representations.

\begin{definition}
	Let $(\sigma,V)$ be a c.b.c. representation of ${E}$ on $\mathcal H.$ A bounded operator ${S}:\mathcal{H} \to E\otimes\mathcal{H}$ is called {\rm generalized inverse} of $\wV$ if ${S}\widetilde{V}{S}={S}$ and $\widetilde{V}{S}\widetilde{V}$=$\widetilde{V}$. For every $n\in \mathbb{N},$ define 
	\begin{align*}
		S^{(n)}= (I_{E^{\otimes n-1}}\ot S)(I_{E^{\otimes n-2}}\ot S)\dots(I_{E}\ot S)S.
	\end{align*}
\end{definition}
% It is not challenging to prove that $(\sigma,V)$ has a generalized inverse if and only if $R(\wV)$ is closed.
The next proposition gives the connection between generalized inverse and generalized range of a regular c.b.c representation.

\begin{proposition}\label{d3}
	Let $(\sigma,V)$ be a regular c.b.c. representation of ${E}$ on $\mathcal H,$  and let $S$ be a generalized inverse of $\wV$. Then $(I_E \ot S){N(\wV)} \subseteq E^{\ot 2} \ot { {{R}^{\infty}{(\widetilde{V})}}} .$
\end{proposition}
\begin{proof}
	Let $\eta \in N(\wV) \subseteq E \ot R(\wV),$ then there exists  $ u \in E^{\ot {2}} \ot \mathcal{H}$ such that $\eta= (I_{E} \ot \widetilde{V})u$. It gives \begin{align*}\widetilde{V}_{2} (I_{E}\otimes S)\eta &= \widetilde{V}_{2} (I_{E} \otimes S\wV)u =\widetilde{V}(I_{E} \otimes \widetilde{V}S\widetilde{V})u = \widetilde{V}(I_{E} \otimes \widetilde{V})u=\widetilde{V}\eta=0.
	\end{align*}
	From Theorem \ref{cc}, we have $(I_E \ot S){N(\wV)} \subseteq N(\wV_2) \subseteq E^{\ot 2} \ot { {{R}^{\infty}{(\widetilde{V})}}}.$
\end{proof} 

The following result from \cite{AHS22} summarizes various properties of a regular c.b.c. repesentation:

\begin{proposition}\label{d2}
	Let $(\sigma,V)$ be a regular c.b.c. representation of ${E}$ on $\mathcal H$ with generalized inverse $S$. Then
	\begin{enumerate}
		\item[(i)] ${{{R}^{\infty}{(\widetilde{V})}}}$ is closed;
		\item[(ii)] $\wV(E \ot {{{R}^{\infty}{(\widetilde{V})}}})={{{R}^{\infty}{(\widetilde{V})}}};$
		\item[(iii)] $S({{{R}^{\infty}{(\widetilde{V})}}}) \subseteq E \ot {{{R}^{\infty}{(\widetilde{V})}}};$
		\item[(iv)] $\wV_nS^{(n)}\wV_n=\wV_n$ for all $n\in \mathbb{N}.$
	\end{enumerate}
\end{proposition}

\begin{proposition}
	Let $(\sigma,V)$ be a c.b.c. representation of $E$ on $\mathcal H$ with closed range, and let $S$ be the generalized inverse of $\widetilde{V}$. Then the following statements are equivalent:
	\begin{enumerate}
		\item[(i)] $(\sigma,V)$ is regular;
		\item[(ii)] $(I_E \ot S^{(k)}){N(\wV)} \subseteq E^{\ot {(k+1)}} \ot { {{R}^{\infty}{(\widetilde{V})}}}$ for all $k\ge 0;$
		\item[(iii)] $(I_E \ot S^{(k)}){N(\wV)} \subseteq E^{\ot {(k+1)}} \ot { {{R}{(\widetilde{V})}}}$ for all $k\ge 0.$
	\end{enumerate}
\end{proposition}
\begin{proof}
	$(i)\Rightarrow (ii)$ It follows from Propositions $\ref{d3}$ and $\ref{d2}$ that \begin{align*}
		(I_E \ot S^{(k)}){N(\wV)}&=(I_{E^{\ot 2}} \ot S^{(k-1)})(I_E\ot S){N(\wV)} \subseteq  {E^{\ot 2}} \ot S^{(k-1)} { {{R}^{\infty}{(\widetilde{V})}}} \\&\subseteq E^{\ot {(k+1)}} \ot { {{R}^{\infty}{(\widetilde{V})}}}.
	\end{align*}
	
	$(ii)\Rightarrow (iii)$ Since ${ {{R}^{\infty}{(\widetilde{V})}}}\subseteq R(\wV),$ we have $(I_E \ot S^{(k)}){N(\wV)} \subseteq E^{\ot {(k+1)}} \ot { {{R}{(\widetilde{V})}}}.$
	
	$(iii)\Rightarrow (i)$ For $n\ge 1$ and $\zeta \in N(\wV_n),$ from \cite[Lemma 5.4]{AHS22}, we get\begin{align*}
		\zeta&=\zeta -S^{(n)}\widetilde{V}_{n}\zeta=\sum_{i=0}^{n-1}(I_{E^{\ot n-i}} \ot S^{(i)})(I_{E^{\ot n-(i+1)}}\ot P_{N(\wV)})(I_{E^{\ot n-i}} \ot \widetilde{V}_i)\zeta.
	\end{align*}
	Since $(I_{E^{\ot n-(i+1)}}\ot P_{N(\wV)})(I_{E^{\ot n-i}} \ot \widetilde{V}_i)\zeta\in {E^{\ot n-(i+1)}}\ot N(\wV)$ and $(I_E \ot S^{(k)}){N(\wV)} \subseteq E^{\ot {(k+1)}} \ot { {{R}{(\widetilde{V})}}}$ for  $k\ge 0,$ we have $\zeta \in E^{\ot n}\ot { {{R}{(\widetilde{V})}}},$ and hence $N(\wV_n) \subseteq E^{\ot n}\ot { {{R}{(\widetilde{V})}}}.$ So from Theorem \ref{cc}, $(\sigma,V)$ is regular.
\end{proof}

Let $(\sigma,V)$ be a c.b.c. representation of $E$ on $\mathcal H$ with closed range. The {\it Moore-Penrose inverse} $\wV^{\dagger}$ of $(\sigma,V)$ \cite{AHS22,BG03,G77} is defined by $$\wV^{\dagger}:=\wV_0^{-1}P_{R(\wV)},$$ where $\wV_0=\wV|_{N(\wV)^{\perp}}: {N(\wV)^{\perp}} \to R(\wV)$ and $P_{R(\wV)}$ is the orthogonal projection of $\mathcal{H}$ onto ${R(\wV)}.$ Equivalently, the  Moore-Penrose inverse of $(\sigma,V)$ is defined  as the unique solution of the following four equations: \begin{equation*}
	\widetilde{V}\widetilde{V}^{\dagger}\widetilde{V}=\widetilde{V},\quad \widetilde{V}^{\dagger}\widetilde{V}\widetilde{V}^{\dagger}=\widetilde{V}^{\dagger},\quad (\widetilde{V}\widetilde{V}^{\dagger})^*=\widetilde{V}\widetilde{V}^{\dagger},\quad( \widetilde{V}^{\dagger}\widetilde{V})^*=\widetilde{V}^{\dagger}\widetilde{V}.
\end{equation*}

\begin{notation}
	For $n\in \mathbb{N},$ define $\widetilde{V}^{\dagger (n)}:\mathcal{H}\to E^{\ot n}\ot \mathcal{H}$ by $$\widetilde{V}^{\dagger(n)}:= (I_{E^{\otimes n-1}}\ot \widetilde{V}^{\dagger})(I_{E^{\otimes n-2}}\ot\widetilde{V}^{\dagger})\dots(I_{E}\ot \widetilde{V}^{\dagger})\widetilde{V}^{\dagger}.$$
\end{notation}

\begin{proposition}\cite{G77,AHS22}
	Suppose that $(\sigma,V)$ is a c.b.c. representation of $E$ on $\mathcal H$ with closed range. Then 
	\begin{enumerate}
		\item  $ R({\widetilde{V}}^\dagger) =R({\widetilde{V}}^*)=N({\widetilde{V})}^{\perp};$   $\:\:\:\:\:\:\:\:\:\:\:\:\:\:\:\:\:\:$   $5.$ $N({\widetilde{V}})=N(\widetilde{V}^\dagger\widetilde{V})=N({\widetilde{V}}^{\dagger*});$
		\item $\widetilde{V}\widetilde{V}^{\dagger}=P_{R({\widetilde{V}})}$,$\widetilde{V}^{\dagger}\widetilde{V}=P_{R({\widetilde{V}^*})};$ $\:\:\:\:\:\:\:\:\:\:\:\:\:$   $6.$  $\widetilde{V}^*\widetilde{V}\widetilde{V}^{\dagger}=\widetilde{V}^\dagger\widetilde{V}\widetilde{V}^*=\widetilde{V}^*;$
		\item$ N(\widetilde{V}^{\dagger})=N(\widetilde{V}\widetilde{V}^{\dagger})=N({\widetilde{V}}^*);$     $\:\:\:\:\:\:\:\:\:\:\:\:\:\:\:$    $7.$ $({\widetilde{V}}^\dagger)^\dagger=\widetilde{V};$ 
		\item$R(\widetilde{V})=R(\widetilde{V}\widetilde{V}^{\dagger})= R({\widetilde{V}}^{\dagger*});$   $\:\:\:\:\:\:\:\:\:\:\:\:\:\:\:\:\:\:$    $8.$ $({\widetilde{V}}^*)^\dagger=({\widetilde{V}}^\dagger)^*.$
	\end{enumerate}
\end{proposition} 

\begin{remark}
	Let $(\sigma,V)$ be a c.b.c. representation of $E$ on $\mathcal H$ with closed range, and let $\wV^{\dagger}$ be the Moore-Penrose inverse of $(\sigma,V).$ Then $(\wV^*\wV)^{\dagger}=\wV^{\dagger}\wV^{*\dagger}.$  Indeed, let $A=\wV^*\wV$ and $B=\wV^{\dagger}\wV^{*\dagger},$ then $ABA=A, BAB=B, (AB)^*=AB, (BA)^*=BA.$ From the uniqueness of the Moore-Penrose inverse, we have $(\wV^*\wV)^{\dagger}=\wV^{\dagger}\wV^{*\dagger}.$
\end{remark}

\begin{proposition}
	Let $(\sigma,V)$ be a regular c.b.c. representation of $E$ on $\mathcal H$ with closed range. Then \begin{enumerate}
		\item $N(\wV^{\dagger (n)})\cap R(\wV_n)=\{0\};$
		\item $R(\wV_n)=\{h : \:\: \wV_n\wV^{\dagger (n)} h=h\}.$
	\end{enumerate}
\end{proposition}
\begin{proof}
	Suppose that $(\sigma,V)$ is regular, then by \cite[Remark 3.7]{AHS22}, we get $\{h : \:\: \wV_n\wV^{\dagger (n)} h=h\}\subseteq R(\wV_n).$ Let $h\in R(\wV_n),$ from Proposition \ref{d2}, we have
	\begin{align*}
		h=P_{R(\wV_n)}h=\wV_n\wV_n^{\dagger}h=(\wV_n\wV^{\dagger(n)}\wV_n)\wV_n^{\dagger}h=\wV_n\wV^{\dagger(n)}(\wV_n\wV_n^{\dagger}h)=\wV_n\wV^{\dagger(n)}h.
	\end{align*}This implies that for $h\in N(\wV^{\dagger (n)})\cap R(\wV_n),$ we get $h=\wV_n\wV^{\dagger(n)}h=0.$
\end{proof}

\section{Wold-type decomposition for bi-regular covariant representations}

In this section, we will study the notion of Cauchy dual of a covariant representation with closed range and a Wold-type decomposition for a bi-regular c.b.c. representation and explore various properties. Let $(\sigma, V)$ be a c.b.c. representation of $E$ on $\mathcal{H}$ with closed range. Define $\wV': E \otimes \mathcal{H} \to \mathcal{H}$ by  $$\wV':=\widetilde{V}(\widetilde{V}^*\widetilde{V})^\dagger.$$ 

Since $\wt{V}^*\wt{V}(\phi(b) \otimes I_{\mathcal{H}})=\wt{V}^*\sigma(b)\wt{V}=(\phi(b) \otimes I_{\mathcal{H}})\wt{V}^* \wt{V}$ for all $b \in \mathcal{B}.$ Pre and post-multiply in the last inequality by $(\wt{V}^*\wt{V})^{\dagger}$, since $(\wt{V}^*\wt{V})^{\dagger}=\wV^{\dagger}\wV^{*\dagger},$ we get $$\wt{V}^{\dagger}\wt{V}(\phi(b) \otimes I_{\mathcal{H}})(\wt{V}^*\wt{V})^{\dagger}=(\wt{V}^*\wt{V})^{\dagger}(\phi(b) \otimes I_{\mathcal{H}})\wt{V}^{\dagger} \wt{V}.$$ Again pre-multiply by $\wV,$ we have 
\begin{equation}\label{U6}
	\sigma(b)\wt{V}'=\wV'(\phi(b) \otimes I_{\mathcal{H}})\wV^{\dagger}\wV.
\end{equation} Since $\wV(\phi(b) \otimes I_{\mathcal{H}})=\sigma(b)\wV,$ simple computations prove that \begin{equation}\label{UG7}
	\wt{V}'(\phi(b) \otimes I_{\mathcal{H}})=\wt{V}^{*\dagger}\wt{V}^{\dagger}\sigma(b)\wV.
\end{equation}
Let $\eta \in E\ot \mathcal{H}=R(\wV^*)\oplus R(\wV^*)^{\perp}=R(\wV^*)\oplus N(\wV),$ then there exist $\eta_1\in R(\wV^*)$ and $\eta_2\in N(\wV)$  such that $\eta=\eta_1+\eta_2.$ Since $N(\wV')=N(\wV),$ from Equations (\ref{U6}) and (\ref{UG7}), we have \begin{align*}
	\sigma(b)\wt{V}'\eta&=\wV'(\phi(b) \otimes I_{\mathcal{H}})\wV^{\dagger}\wV \eta= \wV'(\phi(b) \otimes I_{\mathcal{H}})\eta_1\\&=\wV'(\phi(b) \otimes I_{\mathcal{H}})(\eta_1+\eta_2)=\wV'(\phi(b) \otimes I_{\mathcal{H}})\eta,
\end{align*} and hence $\wt{V}'(\phi(b) \otimes I_{\mathcal{H}})=\sigma(b)\wt{V}'.$ So from Lemma \ref{MSC}, $(\sigma,V')$ is a c.b.c. representation.

\begin{definition}
	We say that  the c.b.c. representation $(\sigma,V')$ defined as above is {\rm Cauchy dual} of $(\sigma,V).$
\end{definition}

\begin{remark}
	Suppose that $(\sigma,V)$ is a left-invertible c.b.c. representation of $E$ on $\mathcal{H}.$ Then $\wt{V}'=\widetilde{V}(\widetilde{V}^*\widetilde{V})^{-1}.$ 
\end{remark}

The next result is a characterization of the Cauchy dual which is an analogue of \cite[Proposition 2.1]{EMZ19}. 

\begin{proposition}\label{R1}
	Let $(\sigma,V)$ be a c.b.c. representation of $E$ on $\mathcal{H}$ with closed range. Then
	\begin{enumerate}
		\item  \label{DS1}	 $\wt{V}'=\widetilde{V}^{\dagger *}=\widetilde{V}^{*\dagger};$    $4.$ $\wt{V}'{^*}\wt{V}'=({\wt{V}^*\wV})';$
		\item $\wt{V}''=\widetilde{V};$   $\:\:\:\:\:\:\:\:\:\:\:\:\:\:\:$   $5.$  $\wt{V}' =\widetilde{V}$ if and only if $\widetilde{V}$ is a partial isometry;
		\item $\wt{V}{^*}'=\wt{V}'{^*}$;     $\:\:\:\:\:\:\:\:\:\:\:$    $6.$ $\widetilde{V}^*\wt{V}'=\wt{V}{^*}'\widetilde{V}=P_{{N}(\widetilde{V})^{\perp}},$ $\wt{V}'\widetilde{V}^*=\widetilde{V}\wt{V}{^*}'=P_{{R}(\widetilde{V})}.$ 
	\end{enumerate}
\end{proposition}

The following proposition gives the Cauchy duals of isomorphic covariant representations are isomorphic:
\begin{proposition}
	Let $(\sigma,V)$ be a c.b.c. representation of $E$ on $\mathcal{H}$ with closed range. If $U$ is a unitary operator on $\mathcal{H},$ then  $$(U^*\wV(I_E\ot U))'=U^*\wV'(I_E\ot U).$$
\end{proposition}
\begin{proof}
	Suppose that $A=(I_E\ot U^*)\wV^*U$ and $B=U^*\wV'(I_E\ot U),$ then $ABA=A, BAB=B, (AB)^*=AB, (BA)^*=BA.$ From uniqueness of the Moore-Penrose inverse, we have $A^{\dagger}=B.$ By Propostion \ref{R1}, we get \begin{align*}
		(U^*\wV(I_E\ot U))'=(U^*\wV(I_E\ot U))^{*\dagger}=((I_E\ot U^*)\wV^*U)^{\dagger}=U^*\wV'(I_E\ot U).
	\end{align*}
\end{proof}
We discuss an example of the Cauchy dual c.b.c. representation.
\begin{example}\label{DHS1}
	For $n\in \mathbb{N},$ we denote $I_n:=\{1,2,\dots,n\}.$ Consider a Hilbert space $\mathcal{H}$ with an orthonormal basis $\{e_m :\: m\in \mathbb{Z}\}$ and a bounded set of  real numbers $\{w_{i,m} : i \in I_n,\:\:m\in \mathbb{Z} \}$ such that $w_{i,0}=0$ and $w_{i,m}\neq 0,m\in \mathbb{Z}\setminus\{0\}$ for all $i \in I_n.$  Suppose $E$ is an $n$-dimensional Hilbert space with the orthonormal basis $\{\delta_i\}_{i\in I_{n}}.$ The bilateral weighted shift c.b.c. representation, $($see \cite[Section 7]{AHS22}$)$, $(\rho, S^{w})$ of $E$ on $\mathcal{H}$ is defined by
	$$S^{w}(\delta_i)=V_i \:\: \mbox{and}\:\:\: \rho(b)=b I_{\mathcal{H}}, \:\:b \in \mathbb{C},$$ where $V_{i}(e_{m})=w_{i,m}e_{nm+i}$ for all $m\in \mathbb{Z}$ and $i\in I_n.$ It is easy to verify that $N(V_i)= span \{e_0\}$, $R(V_i)=\overline{span} \{e_j : \: j\in B_i:=\{nm+i :\: m\in \mathbb{Z}\setminus \{0\}\}\}, R(V_i)\perp R(V_j)$ for distinct $i,j\in I_n$ and  $$V_i^{*}(e_{j})=\begin{cases}
		w_{i,\frac{j-i}{n}}e_{\frac{j-i}{n}} & \text{if }  {j\in B_i}; \\
		0 & \text{if }   {j\notin B_i}.
	\end{cases}$$ Clearly $R(V_i^*)=\{e_m : \: m\in \mathbb{Z}\setminus\{0\}\}.$ Let $x\in N(V_i)^{\perp},$ then $$(\inf_{m\in \mathbb{Z}\setminus \{0\}} |w_{i,m}|) \|x\|\le \|V_ix\| \quad for \quad i\in I_n.$$ It follows that $V_i$ has
	closed range if and only if $\inf \{|w_{i,m}| : w_{i,m}\neq 0\}> 0.$ Now, simple computations show that the Cauchy dual of $V_i$ is $$V_i'(e_m)=V_i^{*\dagger}(e_{m})=\begin{cases}
		\frac{1}{w_{i,m}}e_{nm+i} & \text{if }  {m\neq 0}; \\
		0 & \text{if }   {m=0}.
	\end{cases}$$ Therefore the Cauchy dual $\widetilde{S}^{w '}(\delta_i \ot e_0)=V_i'(e_0)=0$ and $\widetilde{S}^{w '}(\delta_i \ot e_m)=V_i'(e_m)=\frac{1}{w_{i,m}}e_{nm+i}$ for $m\neq 0$ and $i\in I_n.$
\end{example}

Suppose that $(\sigma,V)$ is a left-invertible c.b.c. representation of ${E}$ on $\mathcal H,$ then the Moore-Penrose inverse $\wV^{\dagger}=(\wV^*\wV)^{-1}\wV^*$ is the left inverse of $\wV,$ and hence $\wV$ and $\wV^{\dagger}$ are regular. Therefore, every left-invertible c.b.c. representation is bi-regular. In general, the Moore-Penrose inverse of regular c.b.c. representation need not be regular. We have the following definition:

\begin{definition}
	Let $(\sigma,V)$ be a regular c.b.c. representation of ${E}$ on $\mathcal H.$ We say that $(\sigma,V)$ is {\rm bi-regular} if $N(I_{E^{\ot n}}\ot \wV^{\dagger})\subseteq R(\wV^{\dagger {(n)}})$ for all $n\in \mathbb{N},$ that is, $\wV^{\dagger}$ is regular, where 	$\widetilde{V}^{\dagger(n)}= (I_{E^{\otimes n-1}}\ot \widetilde{V}^{\dagger})(I_{E^{\otimes n-2}}\ot\widetilde{V}^{\dagger})\dots(I_{E}\ot \widetilde{V}^{\dagger})\widetilde{V}^{\dagger}.$ %$N(I_{E^{\ot n}}\ot \wV^{\dagger})$ is the kernel of $(I_{E^{\ot n}}\ot \wV^{\dagger})$ and $R(\wV^{\dagger {(n)}})$ is the range of $\wV^{\dagger {(n)}}.$
\end{definition}
%Note that, if $(\sigma,V)$ is regular c.b.c. representation of ${E}$ on $\mathcal H,$ then $\wV^*$ is also regular.
\begin{remark}\label{PDS1}
	Let $(\sigma,V)$ be a regular c.b.c. representation of ${E}$ on $\mathcal H.$ Then $\wV^*$ is also regular, that is, $N(I_{E^{\ot n}}\ot \wV^{*})\subseteq R(\wV_n^{*})$ for all $n\in \mathbb{N}.$ Indeed, for $n\in \mathbb{N},$ let $\gamma \in R(\wV_n^*)^{\perp}=N(\wV_n).$ Since $(\sigma,V)$ is regular, from Theorem \ref{cc}, $N(\wV_n)\subseteq E^{\ot n}\ot R(\wV).$ There exists $\zeta\in E^{\ot n+1}\ot \mathcal{H}$ such that $(I_{E^{\ot n}}\ot \wV)\zeta=\gamma.$ Let $\eta\in N(I_{E^{\ot n}}\ot \wV^*),$ then $$\langle \eta,\gamma \rangle=\langle \eta,(I_{E^{\ot n}}\ot \wV)\zeta \rangle=\langle (I_{E^{\ot n}}\ot \wV^*)\eta,\zeta \rangle=0.$$ Thus $N(I_{E^{\ot n}}\ot \wV^{*})\subseteq R(\wV_n^{*})$ for all $n\in \mathbb{N}.$
\end{remark}
Now, we discuss some examples of bi-regular c.b.c. representations.
\begin{example} If a partial isometric c.b.c. representation $(\sigma,V)$ of ${E}$ on $\mathcal H$ satisfies $\wV\wV^*\wV=\wV$ and $\wV^*\wV\wV^*=\wV^*,$ then $\wV^*=\wV^{\dagger}.$ So from Remark \ref{PDS1}, regular partial isometric c.b.c. representation is bi-regular.
\end{example}

\begin{example}
	For $n\in \mathbb{N},$ let $\mathcal{H}$ be a Hilbert space with the orthonormal basis $\{e_m :\: m\in \mathbb{Z}\}$ and $E$ be an $n$-dimensional Hilbert space with the orthonormal basis $\{\delta_i\}_{i\in I_{n}}.$ It has been shown in \cite{AHS22} that if bilateral weighted shift c.b.c. representation $(\rho, S^{w})$ of $E$ on $\mathcal{H}$ is regular, then there exists atmost $m_0\in \mathbb{Z}$ such that $w_{i,m_0}=0.$ Let $(\rho, S^{w})$ be a regular c.b.c. representation such that $w_{i,0}=0$ and $w_{i,m}\neq 0,m\in \mathbb{Z}\setminus\{0\}$ for all $i \in I_n.$ Since $V_{i}(e_{m})=w_{i,m}e_{nm+i}$ for all $m\in \mathbb{Z}$ and $i\in I_n.$ It is easy to see that $N(V_i)= span \{e_0\}$, $R(V_i)=\overline{span} \{e_j : \: j\in B_i=\{nm+i :\: m\in \mathbb{Z}\setminus \{0\}\}\}, R(V_i)\perp R(V_j)$ for distinct $i,j\in I_n$ and  $$V_i^{*}(e_{j})=\begin{cases}
		w_{i,\frac{j-i}{n}}e_{\frac{j-i}{n}} & \text{if }  {j\in B_i}; \\
		0 & \text{if }   {j\notin B_i}.
	\end{cases}$$  By definition of the Moore-Penrose inverse of $V_i,$ $$V_i^{\dagger}(e_{j})=\begin{cases}
		\frac{1}{w_{i,\frac{j-i}{n}}}e_{\frac{j-i}{n}} & \text{if }  {j\in B_i}; \\
		0 & \text{if }   {j\notin B_i}.
	\end{cases}$$ This implies that $V_i^{\dagger}(e_{j})=\frac{1}{w_{i,\frac{j-i}{n}}}e_{\frac{j-i}{n}}=\frac{1}{(w_{i,\frac{j-i}{n}})^2}V_i^{*}(e_{j})$ for all $j\in B_i.$ It follows that $V_i^{\dagger}$ is regular for all $i\in I_n.$ Since $R(V_i)\perp R(V_j)$ for distinct $i,j\in I_n,$ we have $(\rho, S^{w})$ is bi-regular.
\end{example}

\subsection{Cauchy dual of bi-regular covariant representations}
A closed $\sigma(\mathcal{B})$-invariant subspace $\mathcal{E}$ of $\mathcal{H}$ is called {\it wandering}  for $(\sigma,V)$ if  $\mathcal{E}$ is orthogonal to  $\wV_n(E^{\ot n}\ot \mathcal{E})$ for all $n \in \mathbb{N}.$ We say that $(\sigma,V)$ has {\it generating wandering subspace property} if
$$\mathcal{H}=[\mathcal{E}]_{\wV}:=\bigvee_{n\geq0}\wV_n(E^{\ot n}\ot \mathcal{E}),$$ and in this case the wandering subspace $\mathcal{E}$ is called {\it generating wandering subspace}.

\begin{definition}
	Let $(\sigma,V)$ be a c.b.c. representation of $E$ on $\mathcal H.$ We say that $(\sigma,V)$ admits {\rm extended Wold-type decomposition} if there exists a wandering subspace $\mathcal{E}$ for $(\sigma,V)$ which decomposes  $\mathcal{H}$ into the direct sum of $(\sigma,V)$-reducing closed subspaces $$\mathcal{H} = [ \mathcal{E}]_{\widetilde{V}} \oplus R^{\infty}({\widetilde{V}})$$ such that $\widetilde{V}\vert_{{E \otimes R^{\infty}{(\widetilde{V})}\cap N({\widetilde{V})}}^\perp}:{E \otimes R^{\infty}{(\widetilde{V})}\cap N({\widetilde{V})^{\perp}}}\rightarrow R^\infty{(\widetilde{V})}$ is a unitary.
\end{definition}

Note that, if $(\sigma,V)$ admits the extended Wold-type decomposition and $\widetilde{V}|_{E\ot R^{\infty}({\widetilde{V}})}$ is unitary or  $\wV$ is one-to-one, then $(\sigma,V)$ admits the Wold-type decomposition (see \cite[Definition 3.1]{HS19}). 

For example, let $(\sigma,V)$ be a regular c.b.c. representation of $E$ on $\mathcal{H}$ which satisfy a growth condition (see \cite[Theorem 5.10]{AHS22}) such that $\wV^{\dagger}$ is a contraction, then $(\sigma,V)$ admits extended Wold-type decomposition. Further, generating wandering subspace is $\mathcal{E}:=\mathcal{H}\ominus \wV(E\ot \mathcal{H}).$

We recall a sufficient condition for bi-regular c.b.c. representations from \cite[Theorem 5.7]{AHS22}.
\begin{theorem}\label{DS2}
	Let $(\sigma,V)$ be a regular c.b.c. representation of $E$ on $\mathcal{H}$ and $R^{\infty}({\widetilde{V}})$ reduces $(\sigma,V).$ Then $(\sigma,V)$ is bi-regular.
\end{theorem}

\begin{proposition}\label{1}
	Let $(\sigma,V)$ be a bi-regular c.b.c. representation of $E$ on $\mathcal H.$ Then $$\mathcal{H} = [ \mathcal{E}]_{\widetilde{V}} \oplus R^{\infty}({\widetilde{V}'})=[\mathcal{E}]_{\wt{V}'} \oplus R^{\infty}({\widetilde{V}}).$$
\end{proposition}
\begin{proof}
	Since $\wV^{\dagger}$ is regular, $\wV'=\wV^{\dagger *}$ is also regular. We know that $\mathcal{E}=R(\wV')^{\perp}$ and $\wV''=\wV,$ then from \cite[Corollary 5.6]{AHS22}, we get
	\[R^{\infty}(\widetilde{V})^{\perp} = [\mathcal{E}]_{\wV'}\quad and \quad R^{\infty}(\wt{V}')^{\perp} = [ \mathcal{E}]_{\wV}. \qedhere\]
\end{proof}

\begin{corollary}\label{DS3}
	Let $(\sigma,V)$ be a c.b.c. representation of $E$ on $\mathcal{H}.$ Let $\wV$ be a left-invertible and $L$ be its left inverse defined by $L=(\wV^*\wV)^{-1}\wV^*.$ Then $$\mathcal{H} = [ \mathcal{E}]_{\widetilde{V}} \oplus R^{\infty}(L^*)=[\mathcal{E}]_{L^*} \oplus R^{\infty}({\widetilde{V}}).$$
\end{corollary}

\begin{corollary}\label{DS4}
	Let $(\sigma,V)$ be a regular c.b.c. representation of $E$ on $\mathcal{H}$ and $R^{\infty}({\widetilde{V}})$ reduces $(\sigma,V).$ Then  $$\mathcal{H} = [ \mathcal{E}]_{\widetilde{V}} \oplus R^{\infty}({\widetilde{V}'})=[\mathcal{E}]_{\wt{V}'} \oplus R^{\infty}({\widetilde{V}}).$$
\end{corollary}

\begin{corollary}\label{DS5}
	Let $(\sigma,V)$ be a regular c.b.c. representation of $E$ on $\mathcal{H}$. Suppose that $R^{\infty}({\widetilde{V}})$ reduces $(\sigma,V)$ and $R^{\infty}({\widetilde{V}'})$ reduces $(\sigma,V').$ Then $[ \mathcal{E}]_{\widetilde{V}}=[\mathcal{E}]_{\wt{V}'}$ and $R^{\infty}({\widetilde{V}'})=R^{\infty}({\widetilde{V}}).$ Moreover,   $\mathcal{H} = [ \mathcal{E}]_{\widetilde{V}} \oplus R^{\infty}({\widetilde{V}}).$
\end{corollary}
\begin{proof}
	It follows from Theorem \ref{DS2} and Corollary \ref{DS4}.
\end{proof}

\begin{remark}\label{123}
	\begin{enumerate}
		\item Let $(\sigma,V)$ be a regular c.b.c. representation of $E$ on $\mathcal{H}.$ If $\wV^*=\wV^{\dagger}$ on $R^{\infty}({\widetilde{V}})$, then $R^{\infty}({\widetilde{V}})$ reduces  $(\sigma,V).$ Thus $(\sigma,V)$ is bi-regular. It is easy to verify that the equality $\wV^*=\wV^{\dagger}$ on $R^{\infty}({\widetilde{V}})$ is equivalent to each one of the following:
		\begin{itemize}
			\item [(i)] $\wV\wV^*P_{R^{\infty}({\widetilde{V}})}=P_{R^{\infty}({\widetilde{V}})};$ 
			\item [(ii)] $P_{R^{\infty}({\widetilde{V}})}\wV\wV^*=P_{R^{\infty}({\widetilde{V}})}\wV\wV^{\dagger};$
			\item [(iii)] $\wV^*P_{R^{\infty}({\widetilde{V}})}=\wV^{\dagger}P_{R^{\infty}({\widetilde{V}})};$
			\item [(iv)]
			$\wV^*\wV P_{E\ot R^{\infty}({\widetilde{V}})}=\wV^{\dagger}\wV P_{E \ot R^{\infty}({\widetilde{V}})};$
			\item [(v)]	
			$\widetilde{V}\vert_{({E \otimes R^{\infty}{(\widetilde{V}))}\cap N({\widetilde{V})}}^\perp}:({E \otimes R^{\infty}{(\widetilde{V})})\cap N({\widetilde{V})^{\perp}}}\rightarrow R^\infty{(\widetilde{V})}$ is a unitary map.
		\end{itemize}
		\item Under any of the assumptions $(i)$-$(v),$ we have $$(\widetilde{V}|_{E \ot {R^{\infty}({\wV})}})^{\dagger} = \widetilde{V}^{\dagger}|_{R^{\infty}({\wV})} = \widetilde{V}^{*}|_{R^{\infty}({\wV})} = (\widetilde{V}|_{E \ot {R^{\infty}({\wV})}})^{*}.$$
	\end{enumerate} 
\end{remark}

Now, we discuss the relation between $\wV'$ and $\wV$ in terms of the extended Wold-type decomposition. 
\begin{proposition}
	Let $(\sigma,V)$ be a bi-regular c.b.c. representation of $E$ on $\mathcal H.$ Then $(\sigma,V)$ admits the extended Wold-type decomposition if and only if $(\sigma,V')$ admits it. In this case, we have $$R^{\infty}({\widetilde{V}})=R^{\infty}(\wt{V}') \quad and \quad [\mathcal{E}]_{\wV}=[\mathcal{E}]_{\wt{V}'}.$$
\end{proposition}
\begin{proof}
	Suppose $(\sigma,V)$ admits the extended Wold-type decomposition, then $\mathcal{H} = [ \mathcal{E}]_{\widetilde{V}} \oplus R^{\infty}({\widetilde{V}}).$ By Proposition \ref{1}, we get $R^{\infty}({\widetilde{V}})=R^{\infty}(\wt{V}')$ and  $[\mathcal{E}]_{\wV}=[\mathcal{E}]_{\wt{V}'}.$
	Since $\wV^{\dagger}$ is regular, $\wt{V}'=\wV^{\dagger *}$ is also regular. From Proposition \ref{d2}, we get $\wt{V}'(E\ot R^{\infty}(\wt{V}'))=R^{\infty}(\wt{V}').$ Let $h\in R^{\infty}(\wt{V}').$ Since $R^{\infty}({\widetilde{V}})=R^{\infty}(\wt{V}')$ and from Proposition \ref{d2} we have 
	\begin{align*}
		\wt{V}'^*h=\wV^{\dagger}h \in \wV^{\dagger}R^{\infty}(\widetilde{V}) \subseteq E \ot R^{\infty}(\widetilde{V})=E \ot R^{\infty}(\wt{V}').
	\end{align*} Therefore $R^{\infty}(\wt{V}')$ reduces $(\sigma,V')$. Next, we want to prove that $$\wt{V}'\vert_{({E \otimes R^{\infty}{(\wt{V}')})\cap N(\wt{V}')}^\perp}:({{E \otimes R^{\infty}{(\wt{V}')})\cap N(\wt{V}')}^\perp}\rightarrow  R^{\infty}(\wt{V}')$$ is a unitary. Let $\eta\in  {{E \otimes R^{\infty}{(\wt{V}')}\cap N(\wt{V}')}^\perp}={{E \otimes R^{\infty}{(\widetilde{V})}\cap N(\widetilde{V})}^\perp}.$ Since $(\sigma,V)$ admits the extended Wold-type decomposition,	$\widetilde{V}\vert_{({E \otimes R^{\infty}{(\widetilde{V})})\cap N({\widetilde{V})}}^\perp}:({E \otimes R^{\infty}{(\widetilde{V})})\cap N({\widetilde{V})^{\perp}}}\rightarrow R^\infty{(\widetilde{V})}$ is a unitary, that is, $\wV^*\wV \eta=\eta.$  If $\wV^*\wV=\wt{V}'^{\dagger}\wV$  and $\wt{V}'\wt{V}'^{\dagger}\wV=\wV,$ then $\wt{V}'\eta=\wt{V}'\wt{V}'^{\dagger}\wV \eta=\wV \eta,$ and hence $\|\wt{V}'\eta\|=\|\wV \eta\|=\|\eta\|$ for every  $\eta\in  ({{E \otimes R^{\infty}{(\wt{V}')})\cap N(\wt{V}')}^\perp}.$ Let $h\in R^{\infty}{(\wt{V}')}=R^{\infty}{(\widetilde{V})},$ by Remark \ref{123}, we get $\wV^*h=\wV^{\dagger}h.$ Therefore $$\|\wt{V}'^*h\|=\|\wV^{\dagger}h\|=\|\wV^*h\|=\|h\|.$$ Thus $\wt{V}'\vert_{({E \otimes R^{\infty}{(\wt{V}')})\cap N(\wt{V}')}^\perp}:{({E \otimes R^{\infty}{(\wt{V}')})\cap N(\wt{V}')}^\perp}\rightarrow  R^{\infty}(\wt{V}')$ is a unitary.
	
	Conversely, suppose that $(\sigma,V')$ admits the extended Wold-type decomposition, by the above discussion $(\sigma,V'')$ also admits the extended Wold-type decomposition. Since $\wV''=\wV,$ and from Lemma \ref{MSC}, $(\sigma,V)$ admits the extended Wold-type decomposition.
\end{proof}

\begin{remark}
	Suppose that $(\sigma,V)$ is a c.b.c. representation of $E$ on $\mathcal{H}$ such that $\wV$ is left-invertible, we get \cite[Corollary 3.9]{HS19} from the above Proposition.
\end{remark}

The following theorem is a generalization of \cite[Theorem 2.4.1]{DJS10} and \cite[Theorem 2]{EMZ21}:
\begin{theorem}
	Let $(\sigma,V)$ be a bi-regular c.b.c. representation of $E$ on $\mathcal H.$ Then the following statements are equivalent:
	\begin{enumerate}
		\item [(i)] $\mathcal{H} \neq [\mathcal{E}]_{\wV};$
		\item [(ii)] $R^{\infty}{(\wt{V}')} \neq \{0\};$
		\item [(iii)] there exists a non-zero closed subspace $\mathcal{M}$ such that $\mathcal{M}\subseteq \wt{V}' (E \ot \mathcal{M});$
		\item [(iv)] there exists a non-zero closed subspace $\mathcal{M}\subseteq R(\wV)$ such that $\wV^* \mathcal{M} \subseteq E\ot \mathcal{M};$
		\item [(v)] there exists a closed subspace $\mathcal{M} \supseteq \mathcal{E}, \mathcal{M}\neq \mathcal{H}$ such that $\wV(E\ot \mathcal{M})\subseteq \mathcal{M};$
		\item [(vi)] there exists a closed subspace $\mathcal{M} \supseteq \mathcal{E}, \mathcal{M}\neq \mathcal{H}$ such that $E \ot \mathcal{M} \subseteq \wV^{\dagger}\mathcal{M}.$
		\item [(vii)] there exist non-zero closed subspaces $\mathcal{M}_1,\mathcal{M}_2$ such that $\wV(E\ot \mathcal{M}_1)\subseteq \mathcal{M}_1, \wV^{\dagger}(\mathcal{M}_1)\subseteq E\ot (\mathcal{M}_1\oplus \mathcal{E}),P_{\mathcal{M}_2}\wV(E\ot \mathcal{M}_2)=\mathcal{M}_2$ and $R(\wV)=\mathcal{M}_1\oplus \mathcal{M}_2$, where $P_{\mathcal{M}_2}$ is the orthogonal projection of $\mathcal{H}$ onto $\mathcal{M}_2.$
	\end{enumerate}
\end{theorem}
\begin{proof}
	From Proposition \ref{1}, $(i) \Leftrightarrow (ii).$
	
	$(ii)\Rightarrow (iii)$ Let $\mathcal{M}=R^{\infty}{(\wt{V}')}\neq \{0\}.$ Since $\wt{V}'$ is regular, $R^{\infty}{(\wt{V}')}$ is closed, and by Proposition \ref{d2}, we have $\wt{V}'(E\ot R^{\infty}(\wt{V}'))=R^{\infty}(\wt{V}').$ Therefore $\mathcal{M}\subseteq \wt{V}'(E \ot \mathcal{M}).$
	
	$(iii)\Rightarrow (ii)$ Let $\mathcal{M}\subseteq \wt{V}'(E \ot \mathcal{M}),$ then $\mathcal{M} \subseteq \wV'_n(E^{\ot n}\ot \mathcal{M})$ for all $n\ge 0,$ and hence $\mathcal{M} \subseteq \bigcap_{n=0}^{\infty} \wV'_n(E^{\ot n}\ot \mathcal{M}).$ It follows that $\{0\}\neq \mathcal{M} \subseteq R^{\infty}{(\wt{V}')}.$
	
	$(ii)\Rightarrow (iv)$ Let $\mathcal{M}=R^{\infty}{(\wt{V}')}\neq \{0\},$ then $\mathcal{M}\subseteq R(\wV).$ Since $\wV'$ is regular and $\wV^*$ is the generalized inverse of $\wV',$ by Proposition \ref{d2} we get $\wV^* (R^{\infty}{(\wt{V}')})\subseteq E\ot R^{\infty}{(\wt{V}')}.$ 
	
	$(iv)\Rightarrow (iii)$ Since $\mathcal{M}\subseteq R(\wV)$ and $\wV'\wV^*=(\wV \wV^{\dagger})^*=\wV \wV^{\dagger}$ is the orthogonal projection of $\mathcal{H}$ onto $R(\wV),$ we get $\mathcal{M}=\wV'\wV^*\mathcal{M}\subseteq  \wV'(E\ot \mathcal{M}).$ 
	
	$(i)\Rightarrow (v)$ Suppose that $\mathcal{M}=[\mathcal{E}]_{\wV},$ then $\mathcal{E} \subseteq \mathcal{M} \neq \mathcal{H}$ and $\wV(E\ot \mathcal{M})\subseteq \mathcal{M}.$ 
	
	$(v)\Rightarrow (iv)$ Since $\mathcal{E}\subseteq \mathcal{M} \neq \mathcal{H},$ we have $\{0\}\neq {\mathcal{M}}^{\perp}\subseteq R(\wV).$ From $\wV(E\ot \mathcal{M})\subseteq \mathcal{M},$ we get $\wV^* \mathcal{M}^{\perp}\subseteq E \ot \mathcal{M}^{\perp}.$ 
	%In particular, $(iv)$ is satisfied with the closed subspace $\mathcal{M}^{\perp}.$
	
	$(i)\Rightarrow (vi)$  Suppose that $\mathcal{M}=[\mathcal{E}]_{\wV},$ then $\mathcal{E} \subseteq \mathcal{M} \neq \mathcal{H}.$ Since $\wt{V}'$ is regular, we have $(E\ot R^{\infty}{(\wt{V}')})^{\perp}\subseteq N(\wt{V}')^{\perp}=R(\wV^{\dagger}).$ So from Proposition \ref{1}, $R^{\infty}{(\wt{V}')}^{\perp}=[\mathcal{E}]_{\wV},$ and hence $E\ot \mathcal{M} \subseteq R(\wV^{\dagger}).$ Since $\wV(E \ot \mathcal{M})\subseteq \mathcal{M}$ and $\wV^{\dagger}\wV=P_{R(\wV^{\dagger})},$ we get $E\ot \mathcal{M}=\wV^{\dagger}\wV(E\ot \mathcal{M})\subseteq \wV^{\dagger}\mathcal{M}.$

	$(vi)\Rightarrow (iii)$ Since $\mathcal{E}\subseteq \mathcal{M} \neq \mathcal{H},$ we have $\{0\}\neq {\mathcal{M}}^{\perp}\subseteq R(\wV).$ We want to prove that $\mathcal{M}^{\perp} \subseteq \wt{V}'(E \ot \mathcal{M}^{\perp}).$ Let $h\in \mathcal{M}^{\perp}\subseteq R(\wV)=R(\wV'),$ then there exists $\eta\in E\ot \mathcal{H}$ such that $h=\wV'\eta$ and $\eta=\eta_1 + \eta_2$ for some $\eta_1\in E\ot \mathcal{M}$ and $\eta_2\in E\ot \mathcal{M}^{\perp}.$ Suppose that $\eta_1\neq 0,$ then for every $m\in \mathcal{M},$ $0=\langle h,m \rangle =\langle \wV'(\eta_1+\eta_2),m \rangle=\langle  (\eta_1+\eta_2),\wV^{\dagger} m \rangle,$ and hence $\eta_1+\eta_2 \in (\wV^{\dagger}\mathcal{M})^{\perp}\subseteq E\ot \mathcal{M}^{\perp}.$ Since $\eta_1+\eta_2 \in E\ot \mathcal{M}^{\perp}$ and $\eta_2 \in E\ot \mathcal{M}^{\perp},$ we have $\eta_1\in E\ot \mathcal{M}^{\perp},$ and thus  $\eta_1=0$ which is a contradiction. It follows that $\mathcal{M}^{\perp} \subseteq \wt{V}'(E \ot \mathcal{M}^{\perp}).$

	$(i)\Rightarrow (vii)$ Let $\mathcal{M}_1=\bigvee_{n=1}\widetilde{V}_{n}(E^{\ot n}\ot \mathcal{E})$ and $\mathcal{M}_2=R^{\infty}{(\wt{V}')}.$ Clearly $\mathcal{M}_1,\mathcal{M}_2\neq \{0\}$ and by Proposition \ref{1}, $\mathcal{M}_1$ and $\mathcal{M}_2$ are orthogonal. If $\mathcal{E}\perp \widetilde{V}_{n}(E^{\ot n}\ot \mathcal{E})$ for every $n\ge 1,$  then
	\begin{align*}
		R(\wV)={\mathcal{E}}^{\perp}=([ \mathcal{E}]_{\widetilde{V}} \oplus R^{\infty}({\widetilde{V}'})) \ominus \mathcal{E}=\mathcal{M}_1\oplus \mathcal{M}_2.
	\end{align*} It is easy to verify that $\wV(E\ot \mathcal{M}_1)\subseteq \mathcal{M}_1.$ Since $\mathcal{M}_1\oplus \mathcal{E}=[\mathcal{E}]_{\wV},$ from Proposition \ref{1}, $(\mathcal{M}_1\oplus \mathcal{E})^{\perp}=R^{\infty}{(\wt{V}')}.$ Using Proposition \ref{d2}, since $\wV'$ is regular, $\wV'(E\ot R^{\infty}{(\wt{V}')}) \subseteq R^{\infty}{(\wt{V}')},$ equivalently, $\wV^{\dagger}(\mathcal{M}_1\oplus \mathcal{E})\subseteq E \ot (\mathcal{M}_1\oplus \mathcal{E}).$ It gives $\wV^{\dagger}(\mathcal{M}_1)\subseteq E \ot (\mathcal{M}_1\oplus \mathcal{E}).$ Next we want to prove that $P_{\mathcal{M}_2}\wV(E\ot \mathcal{M}_2)=\mathcal{M}_2.$ Clearly $P_{\mathcal{M}_2}\wV(E\ot \mathcal{M}_2) \subseteq \mathcal{M}_2,$ let $h\in \mathcal{M}_2=R^{\infty}{(\wt{V}')}\subseteq R(\wV),$ then there exists $\eta\in E\ot \mathcal{H}$ such that $$h=\wV \eta\quad and \quad \eta=\eta_1+\eta_2$$ for some $\eta_1\in E\ot [\mathcal{E}]_{\wV}$ and $\eta_2\in E \ot R^{\infty}{(\wt{V}')})=E \ot \mathcal{M}_2.$ Since $\wV \eta_1\in [\mathcal{E}]_{\wV}\perp \mathcal{M}_2,$ we have $h=P_{\mathcal{M}_2}h=P_{\mathcal{M}_2}\wV \eta_2\in P_{\mathcal{M}_2}\wV(E\ot \mathcal{M}_2).$ 
	
	$(vii)\Rightarrow (v)$ Suppose $(vii)$ holds. First, we want to prove that $\wV(E\ot \mathcal{E})\subseteq \mathcal{M}_1.$ Let $\eta\in E\ot \mathcal{E},$ since $\wV \eta \in R(\wV)=\mathcal{M}_1\oplus  \mathcal{M}_2,$ we have $\wV \eta=h_1+h_2,$ where $h_1 \in \mathcal{M}_1$ and $h_2\in \mathcal{M}_2=P_{\mathcal{M}_2}\wV(E\ot \mathcal{M}_2),$ and hence $h_2=P_{\mathcal{M}_2}\wV \eta_1$ for some $\eta_1\in E\ot \mathcal{M}_2.$ Using orthogonal decomposition of $R(\wV),$ we have $\wV \eta_1=h_3+h_2$ for some $h_3\in \mathcal{M}_1,$ and hence $\wV \eta=h+\wV \eta_1,$ where $h=h_1-h_3\in \mathcal{M}_1.$ It follows that $\wV^{\dagger}\wV \eta=\wV^{\dagger}h+\wV^{\dagger}\wV \eta_1.$ Since $\wV^{\dagger}$ is regular, $\eta\in E\ot \mathcal{E} = E\ot N(\wV^{\dagger})\subseteq R(\wV^{\dagger}),$ and thus $\wV^{\dagger}\wV \eta= \eta =\wV^{\dagger}h+\wV^{\dagger}\wV \eta_1.$ On the other side,  $\wV^{\dagger}h\in \wV^{\dagger}(\mathcal{M}_1)\subseteq E \ot (\mathcal{M}_1 \oplus \mathcal{E})=E \ot {\mathcal{M}_2}^{\perp}$ and $\eta\in E \ot \mathcal{E} \subseteq E \ot {\mathcal{M}_2}^{\perp}.$ This implies that $\wV^{\dagger}\wV \eta_1 \in E \ot {\mathcal{M}_2}^{\perp}=E \ot (\mathcal{M}_1 \oplus \mathcal{E}),$ then there exist $z_1\in E\ot \mathcal{M}_1$ and $z_2 \in E \ot \mathcal{E}$ such that $\wV^{\dagger}\wV \eta_1= z_1 + z_2.$  Now we apply $\wV,$ and we get $\wV \eta_1= \wV z_1 + \wV z_2,$ it gives $z_1 - \eta_1 + z_2 \in N({\wV}) \subseteq E \ot R(\wV).$ Since $z_1- \eta_1 \in E \ot (\mathcal{M}_1 \oplus \mathcal{M}_2)= E \ot R(\wV),$ we have $z_2 \in E \ot R(\wV).$ Therefore $z_2=0,$ and hence $\wV \eta_1= \wV z_1.$ Finally, we get $$\wV \eta = h + \wV z_1 \in \mathcal{M}_1.$$ Since $\wV(E\ot \mathcal{M}_1)\subseteq \mathcal{M}_1,$ we have $\wV(E\ot (\mathcal{M}_1 \oplus \mathcal{E}))\subseteq \mathcal{M}_1 \subseteq \mathcal{M}_1 \oplus \mathcal{E}.$ This implies that a closed subspace $\mathcal{M}:=\mathcal{M}_1 \oplus \mathcal{E}$ and $ \mathcal{M}\neq \mathcal{H}$ such that $\wV(E \ot \mathcal{M})\subseteq \mathcal{M}.$ 
\end{proof}

\section{Powers of Moore-Penrose inverse}

In this section, we will discuss the powers of Moore-Penrose inverse of c.b.c. representations. Suppose $\wV_n^{\dagger}$ is the Moore-Penrose inverse of $\wV_n.$ In general, the equation $\wV_n^{\dagger}=\widetilde{V}^{\dagger(n)},n\ge 2,$ is not true (cf. \cite[Example 4]{EMZ21}). We have the following definition:
\begin{definition}
	Let $(\sigma,V)$ be a c.b.c. representation of $E$ on $\mathcal{H}$ with closed range. For $n\ge 2,$ $(\sigma,V)$ is called {\rm $n$-dagger} if $\wV^{\dagger{(n)}}=\wV_n^{\dagger}.$ Moreover, if $(\sigma,V)$ is $n$-dagger for every $n\ge 2,$ then we say that $(\sigma,V)$ is {\rm hyper-dagger}.
\end{definition}

The next result present the regularity between $\wV$ and $\wV^{\dagger}.$
\begin{proposition}\label{DJ1}
	Let $(\sigma,V)$ be a hyper-dagger c.b.c. representation of $E$ on $\mathcal{H}.$ Then $(\sigma,V)$ is regular if and only if $\wV^{\dagger}$ is regular.
\end{proposition}
\begin{proof}
	Suppose that $(\sigma,V)$ is regular, then $\wV^*$ is also regular. Since $R(\wV)$ is closed, $R(\wV^*)=R(\wV^{\dagger})$ is also closed, we get $E^{\ot n}\ot N(\wV^{\dagger})=E^{\ot n}\ot N(\wV^*)\subseteq R(\wV_n^*)=R(\wV_n^{\dagger})=R(\wV^{\dagger{(n)}})$ for every $n\ge1.$ Thus $\wV^{\dagger}$ is regular. On the other hand, since $(\wV^{\dagger})^{ \dagger}=\wV,$ by given hypothesis, $(\sigma,V)$ is regular.
\end{proof} 

We recall that, let $P$ and $Q$ be two orthogonal projections of $\mathcal{H}$ onto $R(P)$ and $R(Q),$ respectively. Then $PQ=QP=P$ if and only if $R(P)\subseteq R(Q).$ The following theorem is an analogue of \cite[Theorem 3]{EMZ21}:
\begin{theorem}
	Let $(\sigma,V)$ be a regular hyper-dagger c.b.c. representation of $E$ on $\mathcal{H}.$ Then
	\begin{enumerate}
		\item[(1)] $P_{\mathcal{E}}=I-\wV\wV^{\dagger}$ is the orthogonal projection of $\mathcal{H}$ onto $\mathcal{E}=R(\wV)^{\perp};$
		\item[(2)] $\wV_n\wV^{\dagger (n)}$ converges strongly to the orthogonal projection $P,$ where $P$ is onto $R^{\infty}{(\wt{V})};$ 
		\item[(3)] $\sum_{n=0}^{\infty}\wV_n(I_{E^{\ot n}}\ot P_{\mathcal{E}})\wV^{\dagger (n)}$ converges strongly to $Q:= I-P;$
		\item[(4)] $R(Q)=\{h\in \mathcal{H} :\quad \lim_{n\to \infty}\|\wV_n\wV^{\dagger (n)}h\|=0\};$
		\item[(5)] $R(P)$ and $R(Q)$ are $(\sigma,V)$-reducing.
	\end{enumerate}
\end{theorem}
\begin{proof}
	$(1)$ It is trivial.
	%$(1)$ We know that $\wV\wV^{\dagger}$ is the orthogonal projection of $\mathcal{H}$ onto $R(\wV),$ then $I-\wV\wV^{\dagger}$ is the orthogonal projection onto $R(\wV)^{\perp}=\mathcal{E}.$
	
	$(2)$ Let $(\sigma,V)$ be a hyper-dagger, then $P_n:=\wV_n\wV^{\dagger (n)}$ is the orthogonal projection of $\mathcal{H}$ onto $R(\wV_n).$ Since $R(\wV_{n+1})\subseteq R(\wV_n),$ the sequence $(P_n)_{n\ge1}$ converges strongly to the orthogonal projection $P.$ We want to prove $R(P)=R^{\infty}{(\wt{V})}.$ Let $h\in R^{\infty}{(\wt{V})},$ then $P_nh=h$ for all $n\ge 1,$ and hence $Ph=h.$ On the other hand, let $i,j\ge 0$ and $h\in \mathcal{H},$ since $R(\wV_{i+j})\subseteq R(\wV_i),$ we have $P_iP_{i+j}h=P_{i+j}h.$ As $j\to \infty,$ we get $P_iPh=Ph$ for all $i\ge 0.$ It follows that $Ph\in R^{\infty}{(\wt{V})}.$ 
	
	$(3)$  Since $R(\wV_{n+1})\subseteq R(\wV_n),$ we define $Q_n:=P_n-P_{n+1}=\wV_n(I_{E^{\ot n}}\ot P_{\mathcal{E}})\wV^{\dagger (n)}.$ It is easy to verify that $Q_n^*=Q_n$ and $R(Q_n)=R(\wV_{n})\cap R(\wV_{n+1})^{\perp}.$ Since $(\sigma,V)$ is regular, $\wV^*$ is also regular. It gives $E^{\ot n}\ot \mathcal{E}=E^{\ot n}\ot R(\wV)^{\perp}=E^{\ot n}\ot N(\wV^*)\subseteq R(\wV_n^*),$ and thus \begin{align*}
		Q_n^2&=\wV_n(I_{E^{\ot n}}\ot P_{\mathcal{E}})\wV^{\dagger (n)}\wV_n(I_{E^{\ot n}}\ot P_{\mathcal{E}})\wV^{\dagger (n)}\\&=\wV_n(I_{E^{\ot n}}\ot P_{\mathcal{E}})P_{R(\wV_n^*)}(I_{E^{\ot n}}\ot P_{\mathcal{E}})\wV^{\dagger (n)}=\wV_n(I_{E^{\ot n}}\ot P_{\mathcal{E}})\wV^{\dagger (n)}=Q_n.
	\end{align*} Now, we want to prove $R(Q_n)\perp R(Q_m)$ for $n\neq m.$ Suppose $1\le n < m,$ then we have $$R(Q_n)\subseteq R(\wV_{n+1})^{\perp}\subseteq R(\wV_m)^{\perp}\subseteq R(\wV_m)^{\perp}\cup R(\wV_{m+1})=R(Q_m)^{\perp}.$$ Therefore, the sequence $(\sum_{i=0}^{n} Q_i)_{n\ge 0}$  converges strongly to an orthogonal projection $Q.$ Using $R(Q_n)=R(\wV_{n})\cap R(\wV_{n+1})^{\perp}$ for all $n\ge 0$ and $R(Q_n)\perp R(Q_m)$ for $n\neq m,$ we get $R(Q)=\bigoplus_{i=0}^{\infty}(R(\wV_i)\cap R(\wV_{i+1})^{\perp}).$ Using $\sum_{i=0}^{n-1} Q_i=I-P_n,$ we get  $\sum_{n=0}^{\infty}\wV_n(I_{E^{\ot n}}\ot P_{\mathcal{E}})\wV^{\dagger (n)}$ converges strongly to $Q= I-P.$
	
	$(4)$ Since $\sum_{i=0}^{n-1} Q_i=I-P_n=I-\wV_n\wV^{\dagger (n)},$ we have $R(Q)=\{h\in \mathcal{H}: \lim_{n\to \infty}\|\wV_n\wV^{\dagger (n)}h\|=0\}.$
	
	$(5)$ Let $(\sigma,V)$ be a regular, then $\wV({E}\ot R^{\infty}{(\wt{V})})=R^{\infty}{(\wt{V})}.$ On the other hand, let $\eta \in E\ot \mathcal{H}=N(\wV)\oplus N(\wV)^{\perp}=N(\wV)\oplus R(\wV^*)=N(\wV)\oplus R(\wV^{\dagger}),$ then there exist $h\in \mathcal{H}$ and $\xi\in N(\wV)$ such that $\eta=\wV^{\dagger}h+\xi.$ For $m\ge1,$ we have \begin{align*}
		P_{m+1}\wV\eta=\wV_{m+1}\wV^{\dagger (m+1)}\wV\wV^{\dagger}h=\wV_{m+1}\wV^{\dagger (m+1)}h=P_{m+1}h \quad and
	\end{align*} \begin{align*}
		\wV(I_E\ot P_m)\eta&=\wV_{m+1}\wV^{\dagger (m+1)}h+\wV_{m+1}(I_E\ot \wV^{\dagger (m)})\xi\\&=P_{m+1}h+\wV_{m+1}(I_E\ot \wV^{\dagger (m)})\xi.
	\end{align*}Since $\wV^{\dagger}$ is generalized inverse of $\wV$ and $(\sigma,V)$ is regular, using mathematical induction and Proposition \ref{d3}, we get $(I_E\ot \wV^{\dagger (m)})N(\wV)\subseteq N(\wV_{m+1}).$ Therefore $\wV(I_E\ot P_m)\eta=P_{m+1}h,$ and thus $P_{m+1}\wV\eta=\wV(I_E\ot P_m)\eta$ for all $\eta\in E\ot \mathcal{H}.$ As $m\to \infty,$ we obtain $P\wV=\wV(I_E\ot P),$ and thus $R(P)=R^{\infty}{(\wt{V})}$ is $(\sigma,V)$-reducing. And hence $R(Q)=R(P)^{\perp}$ is also $(\sigma,V)$-reducing.
\end{proof}

The next result is an analogue of \cite[Lemma 3.1]{G38}.

\begin{lemma}\label{RK1}
	Let $(\sigma,V)$ be a c.b.c. representation of $E$ on $\mathcal{H}$ with closed range such that $(I_E\ot \wV_i\wV_i^*)N(\wV)\subseteq N(\wV)$ for all $1\le i\le n.$ Then $R(\wV_{i+1})$ is closed for $1\le i\le n.$
\end{lemma}
\begin{proof}
	We will prove it using the mathematical induction. Since $R(\wV)$ is closed, $R(\wV^*)$ is also closed. Suppose that $R(\wV_i^*)$ is closed. We want to prove that $R(\wV_{i+1}^*)$ is closed. Let $\zeta\in \overline{R(\wV_{i+1}^*)}\subseteq \overline{R(I_E \ot \wV_{i}^*)}={R(I_E\ot \wV_{i}^*)},$ then there exists $\eta\in E\ot \mathcal{H}$ such that $\zeta=(I_E\ot\wV_i^*)\eta.$ Since $\eta\in E\ot \mathcal{H}=N(\wV)\oplus N(\wV)^{\perp}=N(\wV)\oplus R(\wV^*),$ we can write $\eta=\xi + \wV^*h$ for some $\xi\in N(\wV)$ and $h\in \mathcal{H}.$ If $(I_E\ot \wV_i\wV_i^*)N(\wV)\subseteq N(\wV)$ for $1\le i\le n,$ then\begin{align*}
		\wV_{i+1}\zeta&=\wV_{i+1}(I_E\ot\wV_i^*)\xi+\wV_{i+1}(I_E\ot\wV_i^*)\wV^*h\\&=	\wV(I_E\ot\wV_i\wV_i^*)\xi+\wV_{i+1}\wV_{i+1}^*h=\wV_{i+1}\wV_{i+1}^*h,
	\end{align*} and hence $\zeta-\wV_{i+1}^*h\in N(\wV_{i+1}).$ Since $\zeta\in \overline{R(\wV_{i+1}^*)}$ and $\wV_{i+1}^*h\in \overline{R(\wV_{i+1}^*)},$ $\zeta-\wV_{i+1}^*h\in \overline{R(\wV_{i+1}^*)}=N(\wV_{i+1})^{\perp}.$ Therefore $\zeta=\wV_{i+1}^*h\in R(\wV_{i+1}^*).$ It follows that $R(\wV_{i+1}^*)$ is closed, and consequently, $R(\wV_{i+1})$ is also closed.
\end{proof}
\begin{theorem}
	Let $(\sigma,V)$ be a c.b.c. representation of $E$ on $\mathcal{H}$ with closed range such that $(I_E\ot \wV_i\wV_i^*)N(\wV)\subseteq N(\wV)$ for all $1\le i\le n.$ Then $\wV_{i+1}^{\dagger}=\wV^{\dagger (i+1)}$ on $R(\wV_{i+1}).$
\end{theorem}
\begin{proof}
	From Lemma \ref{RK1}, $R(\wV_{i+1})$ is closed for $1\le i\le n.$ We will again use mathematical induction, for $i=1,$ since $E^{\ot 2}\ot \mathcal{H}=N(\wV_2)\oplus N(\wV_2)^{\perp}.$ Let $h\in R(\wV_2)\subseteq R(\wV),$ then there exist $\zeta\in N(\wV_2)^{\perp}$ and $\eta\in N(\wV)^{\perp}$ such that $h=\wV_2\zeta$ and $h=\wV \eta.$ It follows that $\eta-(I_E\ot \wV)\zeta\in N(\wV).$ Since $\zeta\in N(\wV_2)^{\perp}=R(\wV_2^*),$ we have $\zeta=(I_E\ot \wV^*)v$ for some $v\in N(\wV)^{\perp},$ and hence $(I_E\ot \wV)\zeta=(I_E\ot \wV\wV^*)v\in (I_E\ot \wV\wV^*)N(\wV)^{\perp} \subseteq N(\wV)^{\perp}.$ Therefore $\eta-(I_E\ot \wV)\zeta\in N(\wV)^{\perp}\cap N(\wV)=\{0\},$ and thus $\eta=(I_E\ot \wV)\zeta.$ Since $N(\wV_2)^{\perp}\subseteq N(I_E\ot \wV)^{\perp}$ and $\wV^{\dagger}\wV=P_{N(\wV)^{\perp}},$ we have \begin{align*}
		\wV^{\dagger (2)}h&=(I_E\ot \wV^{\dagger})\wV^{\dagger}h=(I_E\ot \wV^{\dagger})\wV^{\dagger}\wV \eta=(I_E\ot \wV^{\dagger})\eta=(I_E\ot \wV^{\dagger}\wV)\zeta\\&=P_{N(I_E\ot \wV)^{\perp}}\zeta=\zeta=P_{N(\wV_2)^{\perp}}\zeta=\wV_2^{\dagger}\wV_2\zeta=\wV_2^{\dagger}h.
	\end{align*} Assume that $\wV_{i+1}^{\dagger}=\wV^{\dagger (i+1)}$ on $R(\wV_{i+1}).$ Let $h\in R(\wV_{i+2})\subseteq R(\wV),$ then there exist $\zeta\in N(\wV_{i+2})^{\perp}$ and $\eta\in N(\wV)^{\perp}$ such that $h=\wV_{i+2}\zeta$ and $h=\wV \eta.$ It follows that $\eta-(I_E\ot \wV_{i+1})\zeta\in N(\wV),\wV^{\dagger}h=\eta$ and $$\wV_{i+2}^{\dagger}h=\wV_{i+2}^{\dagger}\wV_{i+2}\zeta=P_{N(\wV_{i+2})^{\perp}}\zeta=\zeta.$$ Since $\zeta\in N(\wV_{i+2})^{\perp}=R(\wV_{i+2}^*),$ we have $\zeta=(I_E\ot \wV_{i+1}^*)v$ for some $v\in N(\wV)^{\perp},$ and thus $(I_E\ot \wV_{i+1})\zeta=(I_E\ot \wV_{i+1}\wV_{i+1}^*)v\in (I_E\ot \wV_{i+1}\wV_{i+1}^*)N(\wV)^{\perp} \subseteq N(\wV)^{\perp}.$ Therefore $\eta-(I_E\ot \wV_{i+1})\zeta\in N(\wV)^{\perp}\cap N(\wV)=\{0\},$ and hence $\eta=(I_E\ot \wV_{i+1})\zeta\in R(I_E\ot \wV_{i+1}).$ Using $\zeta\in N(\wV_{i+2})^{\perp}\subseteq N(I_E\ot \wV_{i+1})^{\perp},\eta\in R(I_E\ot \wV_{i+1})$ and $\wV_{i+1}^{\dagger}=\wV^{\dagger (i+1)}$ on $R(\wV_{i+1}),$ we have  \begin{align*}
		\wV^{\dagger (i+2)}h&=(I_E\ot \wV^{\dagger (i+1)})\wV^{\dagger}h=(I_E\ot \wV^{\dagger (i+1)})\eta=(I_E\ot \wV_{i+1}^{\dagger})\eta\\&=(I_E\ot \wV_{i+1}^{\dagger}\wV_{i+1})\zeta=P_{N(I_E\ot \wV_{i+1})^{\perp}}\zeta=\zeta.
	\end{align*} This completes the proof.
\end{proof}

\section{Cauchy dual of concave covariant representations}

In this section, we will discuss the Cauchy dual of concave completely bounded covariant representations. The following definitions are inspired from \cite[Section 4]{EMZ19}: 
\begin{definition}	
	Let $(\sigma,V)$ be a c.b.c. representation of $E$ on $\mathcal{H}.$ 
	\begin{itemize}
		\item[(i)] The representation $(\sigma,V)$  is said to be {\rm hyponormal modulo} ${N}(\wV)$ if $$\|(I_E \ot \wV^*)\eta\| \le \|\wV \eta\| \quad for\quad every\quad  \eta \in {{N}(\wV)^\perp};$$
		\item [(ii)]  The representation $(\sigma,V)$ is called {\rm $n$-expansive} modulo ${N}(\wV)$ if $$\sum_{j=0}^{n} (-1)^j \binom{n}{j} \|(I_{E^{\ot (n-j)}} \ot \wV_j)\xi\|^2 \le 0  \quad for\quad all\quad   \xi \in {{N}(I_{E^{\ot (n-1)}} \ot \wV)^\perp}.$$
	\end{itemize}
\end{definition}

\begin{definition}	
	Let $(\sigma,V)$ be a c.b.c. representation of $E$ on $\mathcal{H}$ which satisfies
	$$\|\wV_2 \zeta\|^2 + \|\zeta\|^2 \le 2\|(I_E \ot \wV)\zeta\|^2  \quad for\quad all\quad  \zeta \in {{N}(I_E \ot \wV)^\perp}.$$ Then we say that the representation $(\sigma,V)$ is {\rm concave modulo} ${N}(\wV).$
\end{definition}
\begin{example}
	Suppose that $E$ is an $n$-dimensional Hilbert space with the orthonormal basis $\{\delta_i\}_{i\in I_{n}}$ and $\mathcal{H}$ is a Hilbert space with the orthonormal basis $\{e_m :\: m\ge 0\}.$ Let $(\rho, S^{w})$ be the  unilateral weighted shift c.b.c. representation of $E$ on $\mathcal{H}$ defined by
	$$S^{w}(\delta_i)=V_i \:\: \mbox{and}\:\:\: \rho(b)=b I_{\mathcal{H}}, \:\:b \in \mathbb{C},$$ where $V_{i}(e_{m})=w_{i,m}e_{nm+i}$ for all $m\ge 0,i\in I_n$ and $\{w_{i,m} : i \in I_n,\:\:m\ge 0 \}$ is a bounded set of nonnegative real numbers. For $i\in I_n,$ it is easy to verify that $V_i$ is concave if and only if \begin{equation}\label{DD1}
		w_{i,m}^2w_{i,nm+i}^2-2w_{i,m}^2+1\le 0.
	\end{equation}  Let $A=\{m_0\}\subset \mathbb{N}\cup \{0\}$ be a non-empty set. For $i\in I_n,$ we construct a sequence $w_A$ given by $w_A(m)=w_{i,m}$ for $m\notin A$ and $w_A(m_0)=w_{i,{m_0}}=0$. Since $V_{i,A}(e_{m})=w_A{(m)}e_{nm+i},$ we have $N(V_{i,A})=span\{e_{m_0}\}.$ For $i\in I_n,$ we obtain
	\begin{align*}
		\|V^2_{i,A}e_{m_0}\|^2 +\|e_{m_0}\|^2-2\|V_{i,A}e_{m_0}\|^2=w_{i,m_0}^2w_{i,nm_0+i}^2+1-2w_{i,m_0}^2=1> 0,
	\end{align*} and hence $V_{i,A}$ is not concave. From Equation $($\ref{DD1}$)$, $V_{i,A}$ is concave modulo $N(V_{i,A})$ for all $i\in I_n$. Since $R(V_i)\perp R(V_j)$ for distinct $i,j\in I_n,$ we get $N(\widetilde{S}^{w_A})=\bigoplus_{i=1}^{n} span \{\delta_i\ot e_{m_0}\}.$ It is easy to see that $(\rho, S^{w_A})$ is concave modulo $N(\widetilde{S}^{w_A}).$
\end{example}

The next result is an analogue of  \cite[Proposition 4.1]{EMZ19}.
\begin{proposition}
	Let $(\sigma,V)$ be a concave c.b.c. representation modulo ${N}(\wV).$ Then $(\sigma,V)$ has closed range.
\end{proposition}
\begin{proof}
	Let $(\wV \eta_n)\to h$ be a sequence in $R(\wV),$ then there exists $(\eta'_n) \subseteq {N}(\wV)^{\perp}$ such that $\wV \eta_n=\wV \eta'_n.$ This implies that $(\wV \eta'_n)$ is a Cauchy sequence. Since ${{N}(I_E \ot \wV)^\perp}=E \ot {{N}(\wV)}^\perp$ and for every $\xi \in E,$ we have $$ \|\xi \ot \eta'_n\|^2 \le 2\|(I_E \ot \wV)(\xi \ot \eta'_n)\|^2-\|\wV_2( \xi \ot \eta'_n)\|^2 \le 2\|(I_E \ot \wV)(\xi \ot \eta'_n)\|^2.$$ It follows that $(\xi \ot \eta'_n)_{n\ge 0}$ is a Cauchy sequence, and hence $( \eta'_n)_{n\ge 0}$ is also a Cauchy sequence. Suppose that $\lim\limits_{n \to \infty} ( \eta'_n)= \eta',$ then $h=\wV \eta'.$ Thus $R(\wV)$ is closed.
\end{proof}

The following proposition gives the connection between the concavity modulo $N(\wV)$ and the Moore-Penrose inverse: 
\begin{proposition}\label{X3}
	Let $(\sigma,V)$ be a c.b.c. representation of $E$ on $\mathcal{H}$ with closed range. Then $(\sigma,V)$ is concave modulo ${N}(\wV)$ if and only if 
	\begin{align}\label{D1}
		\|\wV_2 \zeta\|^2 + \|(I_E \ot \wV^\dagger \wV)\zeta\|^2 - 2\|(I_E \ot \wV)\zeta\|^2 \le 0 \quad for \quad \zeta \in E^{\ot 2}\ot \mathcal{H}.
	\end{align}
	
\end{proposition}
\begin{proof}
	Let $(\sigma,V)$ be a concave modulo ${N}(\wV)$ and let $\zeta \in E^{\ot 2}\ot \mathcal{H},$ then using $(I_E \ot\wV^\dagger\wV)=P_{{N}(I_E \ot \wV)^\perp}$ and $\wV\wV^\dagger\wV=\wV,$ we have
	\begin{align*}
		&\|\wV_2 \zeta\|^2 + \|(I_E \ot \wV^\dagger \wV)\zeta\|^2 - 2\|(I_E \ot \wV)\zeta\|^2 \\&= \|\wV_2 (I_E \ot \wV^\dagger \wV)\zeta\|^2 + \|(I_E \ot \wV^\dagger \wV)\zeta\|^2 - 2\|(I_E \ot \wV)(I_E \ot \wV^\dagger \wV)\zeta\|^2  \le 0.
	\end{align*}
	On the other side, let $\zeta \in {{N}(I_E \ot \wV)^\perp},$ then $(I_E \ot\wV^\dagger\wV)\zeta=\zeta.$ It follows that $
	\|\wV_2 \zeta\|^2 + \|\zeta\|^2 - 2\|(I_E \ot \wV)\zeta\|^2  \le \|\zeta\|^2-\|(I_E \ot \wV^\dagger \wV)\zeta\|^2=0.$ 
\end{proof}

%We also have the following extension of the case of left invertible covariant representations.
\begin{proposition}\label{X5}
	Let $(\sigma,V)$ be a concave c.b.c. representation modulo ${N}(\wV)$ which satisfies $(I_{E^{\ot n}}\ot \wV) N(I_{E^{\ot n}}\ot \wV)^{\perp} \subseteq N(I_{E^{\ot n-1}}\ot \wV)^{\perp}$ for every $n\ge 1.$ Then  $(\sigma,V)$ is expansive modulo ${N}(\wV).$ Thus $(\sigma,V')$ is contractive.
\end{proposition}
\begin{proof}
	Using \cite[Lemma 2.2]{HS19}, for a representation $(\sigma,V)$ which is concave modulo ${N}(\wV),$ we get
	\begin{align*} 
		\|\wt{V}_2\zeta\|^2-\|\zeta\|^2 \le 2(\|(I_{E} \ot \wV)\zeta\|^2-\|\zeta\|^2) \quad for \quad \zeta \in {{N}(I_E \ot \wV)^\perp}.
	\end{align*}
	For $n\ge1$ and $y \in {{N}(I_{E^{\ot (n-1)}} \ot \wV)^\perp},$ using the  mathematical induction, we have	\begin{align*}
		\|\wt{V}_ny\|^2-\|y\|^2 \le n(\|(I_{E^{\ot (n-1)}} \ot \wV)y\|^2-\|y\|^2) .
	\end{align*}
	This implies that  $\|y\|^2+n(\|(I_{E^{\otimes (n-1)}} \otimes \wt{V})y\|^2-\|y\|^2) \geq 0.$ 
	Thus $\|(I_{E^{\otimes (n-1)}} \otimes \wt{V})y\|^2\geq \frac{n-1}{n} \|y\|^2. $ From the properties of the creation operators, we obtain 
	$$\|\wt{V}\xi\| \geq \frac{n-1}{n}\|\xi\|~\mbox{for all}~\xi \in {{N}( \wV)^\perp}.$$ As $n \rightarrow \infty$ we get  $\|\xi\|\leq \| \wt{V}\xi\|,$ and hence  $(\sigma,V)$ is expansive modulo ${N}(\wV)$ and $\gamma(\wV)\ge 1.$ So from \cite[Proposition 4.5]{AHS22}, $(\sigma,V')$ is contractive.
\end{proof}
The following theorem is the main result of this section which is a gen-
eralization of \cite[Theorem 4.1]{EMZ19}:
\begin{theorem}\label{X4}
	Let $(\sigma,V)$ be a concave c.b.c. representation modulo ${N}(\wV).$ Then the Cauchy dual $(\sigma,V')$ is hyponormal modulo ${N}(\wV).$
\end{theorem}
\begin{proof}
	Let $\zeta \in E^{\ot 2}\ot \mathcal{H},$ we have
	\begin{align*}
		&\|\wV_2\zeta\|^2 - \|(I_E \ot \wV^* \wV)\zeta\|^2 \le \|\wV_2\zeta\|^2 - \|(I_E \ot \wV^* \wV)\zeta\|^2 + \\& \quad \|(I_{E^{\ot 2}\ot \mathcal{H}}-(I_E \ot \wV^* \wV))\zeta\|^2 = \|\wV_2\zeta\|^2  + \|\zeta\|^2 - 2\|(I_E \ot \wV)\zeta\|^2.
	\end{align*} 
	For  $\zeta\in {{N}(I_E \ot \wV)^\perp},$ since $(\sigma,V)$ is concave modulo ${N}(\wV),$ we have $\|\wV_2 \zeta\| \le \|(I_E \ot \wV^* \wV)\zeta\|.$
	For every $\eta \in E^{\ot 2}\ot \mathcal{H}={{N}(I_E \ot \wV)}\oplus {{N}(I_E \ot \wV)^\perp},$ there exist $\eta_1\in {N}(I_E \ot \wV)$ and $\zeta\in {{N}(I_E \ot \wV)^\perp}$ such that $\eta=\eta_1 +\zeta.$ It follows that  $$\|\wV_2\eta\|=\|\wV_2\zeta\| \le \|(I_E \ot \wV^* \wV)\zeta\|=\|(I_E \ot \wV^* \wV)\eta\|\quad for \quad \eta \in E^{\ot 2}\ot \mathcal{H}.$$ Since  $R(\wV)$ is closed, $\wV^{\dagger}$ exists. From Proposition \ref{R1}, we have $$\wV(I_E \ot \wV'\wV^*\wV)=\wV(I_E \ot P_{R(\wV)}\wV)=\wV_2,$$ and hence $\|\wV(I_E \ot \wV'\wV^*\wV)\eta\|=\|\wV_2\eta\|\le \|(I_E \ot \wV^* \wV)\eta\|$ for $\eta \in E^{\ot 2}\ot \mathcal{H}.$ For any arbitary $\eta \in E^{\ot 2}\ot \mathcal{H}={{N}(I_E \ot \wV)}\oplus R(I_E\ot \wV^*\wV),$ there exist $\zeta_1\in {N}(I_E \ot \wV)$ and $\zeta_2 \in E^{\ot 2}\ot \mathcal{H}$ such that $\eta=\zeta_1 + (I_E\ot \wV^*\wV)\zeta_2.$ Using ${N}(I_E\ot \wV')={N}(I_E \ot \wV),$ we get $$\|\wV(I_E \ot \wV' )\eta\|=\|\wV(I_E \ot \wV'\wV^*\wV)\zeta_2\| \le \|(I_E \ot \wV^* \wV)\zeta_2\| \le \|\eta\|.$$ Let $\wT=\wV(I_E \ot \wV').$ For $\eta\in E^{\ot 2}\ot \mathcal{H}$ and $b\in \mathcal{B}$ we have \begin{align*}
		\wT(\phi_2(b)\ot I_{\mathcal{H}})\eta&=\wV(\phi(b)\ot I_{\mathcal{H}})(I_E \ot \wV')\eta=\sigma(b)\wT\eta,
	\end{align*} where $\phi_2$ is the left action of $\mathcal{B}$ on $E^{\ot 2}.$ So from Lemma \ref{MSC}, $(\sigma,T)$ is a completely contractive covariant representation of $E^{\ot 2}$ on $\mathcal{H}.$ Using Proposition \ref{R1}, we get $$\wT^*\wV'\wV^*=(I_E \ot \wV')^*\wV^*\wV'\wV^*=(I_E \ot \wV')^*P_{R(\wV^*)}\wV^*=(I_E \ot \wV')^*\wV^*.$$ It gives $\|(I_E \ot \wV')^*\wV^*h\|=\|\wT^*\wV'\wV^*h\| \le \|\wV'\wV^*h\|$ for every $h\in\mathcal{H},$ thus $(\sigma,V')$ is hyponormal modulo ${N}(\wV).$
\end{proof}
The next result is an application of Proposition \ref{X5} and Theorem \ref{X4} which is a generalization of \cite[Theorem 2.9]{SC07}. 
\begin{corollary}\label{Y1}
	Let	$(\sigma,V)$ be a concave c.b.c. representation of $E$ on $\mathcal{H}.$ Then the Cauchy dual $(\sigma,V')$ is hyponormal contractive covariant representation.
\end{corollary}

\begin{corollary}
	Let $(\sigma,V)$ be a c.b.c. representation of $E$ on $\mathcal{H}$ such that \begin{equation*}
		\|(I_E\ot \wV)\zeta + \eta\|^2 \le 2(\|\zeta\|^2 + \|\wV \eta\|^2), \quad for \quad \zeta \in E^{\ot 2}\ot \mathcal{H}, \quad  \eta \in E\ot \mathcal{H}. 
	\end{equation*}
	Then $(\sigma,V)$ is hyponormal contractive.
\end{corollary}
\begin{proof}
	Suppose that $(\sigma,V)$ satisfies $\|(I_E\ot \wV)\zeta + \eta\|^2 \le 2(\|\zeta\|^2 + \|\wV \eta\|^2)$ for $\zeta \in E^{\ot 2}\ot \mathcal{H}$ and  $\eta \in E\ot \mathcal{H},$ then by \cite[Theorem 3.13]{HS19}, $(\sigma,V')$ is concave. Since $\wV''=\wV,$ we have $(\sigma,V)$ is hyponormal contractive.
\end{proof}

% \begin{corollary}
%	Let	 $(\sigma, V)$ be a c.b.c. representation of $E$ on a Hilbert space $\mathcal{H}.$ Then $(\sigma, V')$ is a hyponormal contraction.
%\end{corollary}

\begin{corollary}
	Let $(\sigma, V)$ be a c.b.c. representation of $E$ on $\mathcal{H}$ with closed range such that 
	\begin{equation}\label{X1}
		\|(I_E\ot \wV)\zeta + (I_E\ot \wV \wV^{\dagger})\wV^{\dagger}\wV \eta\|^2 \le 2(\|\zeta\|^2+\|\wV \eta\|^2)
	\end{equation} for all $\zeta \in E^{\ot 2}\ot \mathcal{H}$ and $\eta\in E\ot \mathcal{H}.$ Then $(\sigma,V)$ is hyponormal modulo ${N}(\wV).$
\end{corollary}
\begin{proof}
	Suppose $(\sigma, V)$ satisfies $\|(I_E\ot \wV)\zeta + (I_E\ot \wV \wV^{\dagger})\wV^{\dagger}\wV \eta\|^2 \le 2(\|\zeta\|^2+\|\wV \eta\|^2).$
	Put $h=\wV \eta$ we have \begin{align}\label{RK3}
		\|(I_E\ot \wV)\zeta +(I_E\ot \wV \wV^{\dagger})\wV^{\dagger}h\|^2 \le 2(\|\zeta\|^2+\|h\|^2)
	\end{align} for $\zeta \in E^{\ot 2}\ot \mathcal{H}$ and $h\in R(\wV).$ Define an operator $X: (E^{\ot 2}\ot \mathcal{H}) \oplus  \mathcal{H}\to E\ot \mathcal{H}$ by $X(\zeta,h)=(I_E\ot \wV)\zeta + (I_E\ot \wV \wV^{\dagger})\wV^{\dagger}h$ for all $\zeta \in E^{\ot 2}\ot \mathcal{H}$ and $h\in\mathcal{H}.$ From Equation (\ref{RK3}), it is easy to verify that $\|X\|\le \sqrt{2},$ and hence $XX^*\le 2 I_{E\ot \mathcal{H}},$ which yields
	\begin{equation*}\label{Y8}
		(I_E \ot \wV\wV^*)+ (I_E \ot \wV\wV^{\dagger})\wV^{\dagger}\wV' (I_E \ot \wV\wV^{\dagger}) \le 2 I_{E\ot \mathcal{H}}.
	\end{equation*}
	Pre and post-multiply in the last inequality by $(I_E \ot \wV^{\dagger})$ and  $(I_E \ot \wV'),$ respectively, we obtain
	\begin{equation*}
		(I_E \ot \wV^{\dagger}\wV\wV^*\wV')+ (I_E \ot \wV^{\dagger}\wV\wV^{\dagger})\wV^{\dagger}\wV' (I_E \ot \wV\wV^{\dagger}\wV') \le 2 (I_E \ot \wV^{\dagger}\wV').
	\end{equation*}
	Since $\wV^*\wV'=\wV^{\dagger}\wV=P_{R(\wV^*)}$ and $\wV\wV^{\dagger}\wV'=\wV'$ we get $$(I_E \ot \wV^*\wV')+ (I_E \ot \wV^{\dagger})\wV^{\dagger}\wV_2' \le 2(I_E \ot \wV^{\dagger}\wV').$$ It gives
	$\|\wV_2'\zeta\|^2 -2\|(I_E\ot \wV')\zeta\|^2 +\|(I_E\ot \wV^*\wV')\zeta\|^2\le 0$ for all $\zeta \in E^{\ot 2}\ot \mathcal{H}.$ So by Proposition \ref{X3}, $(\sigma,V')$ is concave modulo ${N}(\wV')={N}(\wV).$ From Theorem \ref{X4}, $(\sigma,V)$ is hyponormal modulo ${N}(\wV).$
\end{proof}

\subsection*{Acknowledgment}
The author is grateful to the reviewer for carefully reading the manuscript and giving valuable suggestions and comments. He want to thank Harsh Trivedi and Shankar Veerabathiran for some fruitful discussions. The author supported by UGC fellowship (File No:16-6(DEC. 2018)/2019(NET/CSIR)) and acknowledge the Centre for Mathematical \& Financial Computing and the DST-FIST grant for the financial support for the computing lab facility under the scheme FIST ( File No: SR/FST/MS-I/2018/24) at the LNMIIT, Jaipur.

\end{document}